\documentclass[a4paper, 11pt]{amsart}
\usepackage{amsmath,amsthm,amssymb,enumerate}
\usepackage[arrow,matrix]{xy}
\usepackage{mathtools}
\usepackage{color}
\usepackage{hyperref}
\usepackage{tikz}
\usetikzlibrary{shapes.geometric}

\theoremstyle{plain}
\newtheorem{theorem}{Theorem}[section]
\newtheorem{proposition}[theorem]{Proposition}
\newtheorem{lemma}[theorem]{Lemma}
\newtheorem{corollary}[theorem]{Corollary}
\numberwithin{equation}{section}

\theoremstyle{definition}
\newtheorem{definition}[theorem]{Definition}
\newtheorem{remark}[theorem]{Remark}
\newtheorem{example}[theorem]{Example}
\newtheorem{question}[theorem]{Question}
\newtheorem{convention}[theorem]{Convention}

\newcommand{\C}{\mathbb{C}}

\newcommand{\R}{\mathbb{R}}
\newcommand{\Z}{\mathbb{Z}}

\newcommand{\CP}{\mathbb{C}P}

\newcommand{\cF}{\mathcal{F}}

\newcommand{\N}{\mathbb{N}}

\def\red#1{{\textcolor{red}{#1}}}

\DeclareMathOperator{\wed}{wedge}
\DeclareMathOperator{\link}{Lk}
\DeclareMathOperator{\proj}{Proj}
\DeclareMathOperator{\relint}{relint}
\DeclareMathOperator{\pos}{pos}
\DeclareMathOperator{\Pic}{Pic}
\DeclareMathOperator{\conv}{conv}

\begin{document}
\title[Wedge operations and torus symmetries]{Wedge operations and torus symmetries}

\author[S.Choi]{Suyoung Choi}
\address{Department of Mathematics, Ajou University, San 5, Woncheondong, Yeongtonggu, Suwon 443-749, Korea}
\email{schoi@ajou.ac.kr}

\author[H.Park]{Hanchul Park}
\address{Department of Mathematics, Ajou University, San 5, Woncheondong, Yeongtonggu, Suwon 443-749, Korea}
\email{hpark@ajou.ac.kr}

\thanks{The authors were partially supported by Basic Science Research Program through the National Research Foundation of Korea(NRF) funded by the Ministry of Education, Science and Technology(NRF-2012R1A1A2044990). The first author was additionally supported by TJ Park Science Fellowship funded by POSCO TJ Park Foundation.}

\subjclass[2010]{14M25, 52B20, 52B35}
\keywords{toric variety, projective toric variety, Gale diagram, simplicial wedge, topological toric manifold, real topological toric manifold, quasitoric manifold, small cover, real toric variety}

\date{\today}

\begin{abstract}
    A fundamental result of toric geometry is that there is a bijection between toric varieties and fans.
    More generally, it is known that some class of manifolds having well-behaved torus actions, called topological toric manifolds $M^{2n}$, can be classified in terms of combinatorial data containing simplicial complexes with $m$ vertices.
    We remark that topological toric manifolds are a generalization of smooth toric varieties.
    The number $m-n$ is known as the Picard number when $M^{2n}$ is a {compact smooth} toric variety.

    In this paper, we investigate the relationship between the topological toric manifolds over a simplicial complex $K$ and those over the complex obtained by simplicial wedge operations from $K$. As applications, we do the following.
    \begin{enumerate}
        \item We classify smooth toric varieties of Picard number $3$. This is a reproving of a result of Batyrev.
        \item We give a new and complete proof of projectivity of smooth toric varieties of Picard number $3$ originally proved by Kleinschmidt and Sturmfels.
        \item We find a criterion for a toric variety over the join of boundaries of simplices to be projective. When the toric variety is smooth, it is known as a generalized Bott manifold which is always projective.
        \item We classify and enumerate real topological toric manifolds when $m-n=3$. In particular, when $P$ is a polytope whose Gale diagram is a pentagon with assigned numbers $(a_1,a_3,a_5,a_2,a_4)$, then every real topological toric manifold over $P$ is a real toric variety, and the number \#DJ of them up to Davis-Januszkiewicz equivalence is
        \[
             \hbox{\#DJ}= 2^{a_1+a_4-1}+2^{a_2+a_5-1}+2^{a_3+a_1-1}+2^{a_4+a_2-1}+2^{a_5+a_3-1}-5.
         \] When $P$ is a polytope whose Gale diagram is a heptagon with arbitrary assigned numbers, no real topological toric manifold over $P$ is a real toric variety, and we have $\#DJ=2$.
        \item When $m-n \leq 3$, any real topological toric manifold is realizable as fixed points of the conjugation of a topological toric manifold.
    \end{enumerate}

\end{abstract}

\maketitle

\tableofcontents

\section{Introduction}

    A \emph{toric variety}, which arose in the field of algebraic geometry, of dimension $n$ is a normal algebraic variety with an action of an algebraic torus $(\C^\ast)^n$ having a dense orbit. A compact smooth toric variety is sometimes called a \emph{toric manifold}. By regarding $S^1$ as the unit circle in $\C^\ast$, there is a natural action of $T^n = (S^1)^n\subset (\C^\ast)^n$ on a toric variety. Instead of an algebraic torus action on an algebraic variety, one could think of a smooth torus ($T^n$ or $(\C^\ast)^n$) action on a smooth manifold. Since the pioneering work of Davis and Januszkiewicz \cite{DJ91}, a number of categories of manifolds which admit certain torus actions have been proposed as topological analogues of smooth toric varieties.

    A \emph{torus manifold} introduced in \cite{HM03} is a closed smooth orientable manifold of dimension $2n$ which admits an effective $T^n$-action with the non-empty fixed points set.
    Since every toric manifold admits a $T^n$-action, any compact smooth  toric variety is a torus manifold.

    A \emph{quasitoric manifold}\footnote{The authors would like to indicate that the notion of quasitoric manifolds originally appeared under the name ``toric manifolds'' in \cite{DJ91}. Later, it was renamed in \cite{BP} in order to avoid confusion with smooth compact toric varieties.} introduced in \cite{DJ91} is a closed smooth $2n$-manifold $M$ with an effective $T^n$-action such that
    \begin{enumerate}
      \item the torus action is \emph{locally standard}: i.e., it is locally isomorphic to the standard action of $T^n$ on $\R^{2n}$,
      \item the orbit space $M/T^n$ can be identified with a simple polytope $P^n$.
    \end{enumerate}
    A quasitoric manifold is surely a torus manifold. Moreover, every smooth projective toric variety is a quasitoric manifold.
    As far as the authors know, there is no known example of a (non-projective) toric manifold whose $T^n$-orbit is not a simple polytope. In other words, every known example of toric manifolds is a quasitoric manifold.

    A \emph{topological toric manifold} defined in \cite{IFM12} is a closed smooth $2n$-manifold $M$ with an effective smooth $(\C^\ast)^n$-action such that there is an open and dense orbit and $M$ is covered by finitely many invariant open subsets each of which is equivariantly diffeomorphic to a smooth representation space of $(\C^\ast)^n$. Every topological toric manifold is a torus manifold. Furthermore, every toric manifold and quasitoric manifold is a topological toric manifold by \cite{IFM12}. Therefore, we obtain a diagram of inclusions of manifolds equipped with torus actions:


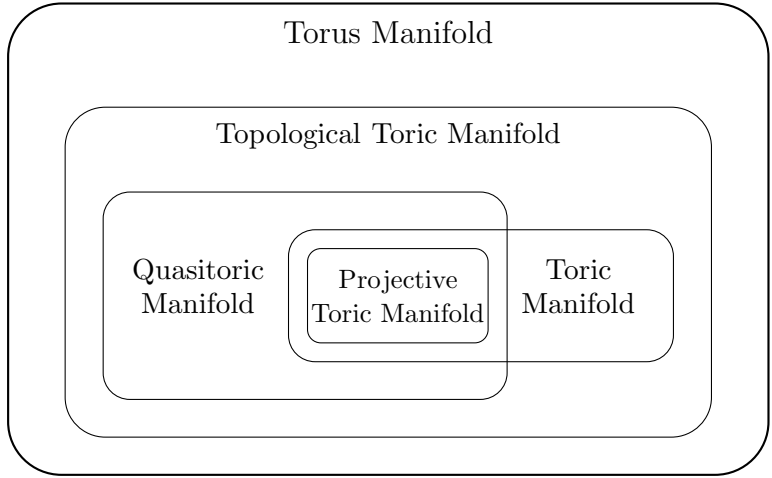
{
\begin{center}
\begin{tikzpicture}[scale=0.125]
    \draw[thick, rounded corners=20pt] (0,0) rectangle (80,50);
    \draw (40,47) node {\large{Torus Manifold}};
    \draw[rounded corners=15pt] (6,4) rectangle (74,39);
    \draw (40,36) node {Topological Toric Manifold};
    \draw[rounded corners=10pt] (10,8) rectangle (52.5,30);
    \draw (20,20) node [align=center]{Quasitoric \\Manifold};
    \draw[rounded corners=10pt] (29.5,12) rectangle (70,26);
    \draw (60,20) node [align=center]{Toric \\Manifold};
    \draw[rounded corners=5pt] (31.5,14) rectangle (50.5,24);
    \draw (41,19) node [align=center]{\small Projective \\ \small Toric Manifold};
\end{tikzpicture}
\end{center}
}

    The theory of toric varieties has been grown up very highly for the last decades. One of the most important results for toric varieties is that there is a bijection between toric varieties and fans. Roughly speaking, a fan is a collection of strongly convex rational cones in $\R^n$ satisfying that each face of cones and each intersection of a finite number of cones are also in the fan. A fan is said to be \emph{complete} if the union of all cones covers the whole space $\R^n$, and is said to be \emph{non-singular} if one-dimensional faces (simply, rays) of each cone are unimodular in $\Z^n$. It is known that a toric variety is compact (resp. smooth) if and only if its corresponding fan is complete (resp. non-singular). Therefore, there is a bijection between toric manifolds and complete non-singular fans.

    We note that a complete non-singular fan can be regarded as a pair of a simplicial complex and the data of rays. That is, for a given complete non-singular fan $\Sigma$ of dimension $n$, one obtains a pair $(K,\lambda)$, where $K$ is the face complex of $\Sigma$ and $\lambda$ is the map which assigns the primitive integral vector in $\Z^n$ representing a ray of $\Sigma$ to the corresponding vertex of $K$. Such a pair $(K,\lambda)$ is called a \emph{characteristic map} of dimension $n$.

    Topological toric manifolds (and quasitoric manifolds) also have their characteristic maps characterizing them, enabling the notation $M = M(K,\lambda)$. {A \emph{characteristic map} of dimension $n$ is defined as the pair $(K,\lambda)$ of an abstract simplicial complex $K$ of dimension $\le n-1$ and a map $\lambda\colon V(K)\to \Z^n$ so that $\{\lambda(i)\mid i\in \sigma\}$ is a linearly independent set over $\R$ for any face $\sigma$ of $K$, where $V(K) =[m]$ is the vertex set of $K$.}
    Meanwhile, when $M$ is a quasitoric manifold, $K$ is the face complex of a simplicial polytope. Such a simplicial complex is said to be \emph{polytopal}. When $M$ is a topological toric manifold, $K$ is the underlying simplicial complex of a complete fan. Such a simplicial complex is said to be \emph{fan-like}.

    The completeness and non-singularity of the characteristic maps are defined similarly. We emphasize that the characteristic map is a useful tool connecting topology of manifolds and combinatorics of the underlying simplicial complexes.

    There is a classical operation of simplicial complexes called the \emph{simplicial wedge operation} (refer \cite{PB80} for example).
    As shown by Bahri-Bendersky-Cohen-Gitler in \cite{BBCG10}, it is deeply related with polyhedral products and generalized moment angle complexes and is gaining more interests in the field of toric theory.
    Let $K$ be a simplicial complex with $m$ vertices and fix a vertex $v$. Consider a 1-simplex $I$ whose vertices are $v_{1}$ and $v_{2}$ and denote by $\partial I$ the 0-skeleton of $I$. Now, let us define a new simplicial complex on $m+1$ vertices, called the \emph{(simplicial) wedge} of $K$ at $v$, denoted by $\wed_{v}(K)$, by
    \[ \wed_{v}(K)= (I \star \link_K\{v\}) \cup (\partial I \star (K\setminus\{v\})), \]
    where $K\setminus\{v\}$ is the induced subcomplex with $m-1$ vertices except $v$, the $\link_K\{v\}$ is the link of $v$ in $K$,
    and $\star$ is the join operation of simplicial complexes.
    Let $P$ be a simple polytope whose face structure is $K$. A wedge of $P$ is defined as the simple polytope whose face structure is isomorphic to a wedge of $K$.
    For a given characteristic map $(K,\lambda)$ and the associated topological toric manifold $M$, there is a natural construction of a new topological toric manifold whose underlying complex is a wedge $\wed_v(K)$.
    In his paper \cite{Ewa86}, Ewald introduced the construction for toric varieties and he called it the \emph{canonical extension}. In the paper \cite{BBCG10}, the authors rediscovered the idea and defined essentially the same manifold using the notation $M(J)$ in the category of quasitoric manifolds.
    {
    Let $(K,\lambda)$ be a characteristic map of dimension $n$ and $\sigma$ a face of $K$ such that the vectors $\lambda(i)$, $i\in\sigma$, are unimodular. Then a characteristic map $(\link_K\sigma, \proj_\sigma \lambda)$, called the \emph{projected characteristic map}, is defined by the map
    \[
        (\proj_\sigma \lambda )(v) = [\lambda(v)]\in \Z^n/\langle \lambda(w)\mid w\in\sigma\rangle \cong \Z^{n - |\sigma|}.
    \]
    }
    The canonical extension of $(K,\lambda)$ is determined by the characteristic map $(\wed_{v}(K),\lambda')$ where $\proj_{v_1}\lambda' = \proj_{v_2}\lambda' = \lambda$.

    Generalizing this concept, we try to find every non-singular characteristic map whose underlying simplicial complex is $\wed_{v}(K)$. Let $K$ be a fan-like simplicial sphere of dimension $n-1$ equipped with an orientation $o$ as a simplicial manifold. Then the characteristic map $(K,\lambda)$ is said to be \emph{positively oriented} if the sign of $\det(\lambda(i_1),\dotsc,\lambda(i_n))$ coincides with $o(\sigma)$ for any oriented maximal simplex $\sigma=(i_1, \ldots, i_n)\in K$.

    The following is our main result:
    \begin{theorem}\label{thm:mainthm}
        Let $K$ be a fan-like simplicial sphere and $v$ a given vertex of $K$. Let  $(\wed_v(K), \lambda)$ be a characteristic map and let $v_1$ and $v_2$ be the two new vertices of $\wed_v(K)$ created from the wedging. {Let us assume that $\{\lambda(v_1),\lambda(v_2)\}$ is a unimodular set.} Then $\lambda$ is uniquely determined by the projections $\proj_{v_1}\lambda$ and $\proj_{v_2}\lambda$. Furthermore,
        \begin{enumerate}
            \item $\lambda$ is non-singular if and only if so are $\proj_{v_1}\lambda$ and $\proj_{v_2}\lambda$.
            \item $\lambda$ is positively oriented if and only if so are $\proj_{v_1}\lambda$ and $\proj_{v_2}\lambda$.
            \item $\lambda$ is fan-giving if and only if so are $\proj_{v_1}\lambda$ and $\proj_{v_2}\lambda$.
        \end{enumerate}
    \end{theorem}

    Combining this with the fact that $K$ is polytopal if and only if its wedge is polytopal, we can say:
    \begin{itemize}
    \item If one knows every topological toric manifold over $K$, then we know every topological toric manifold over a wedge of $K$.
    \item If one knows every quasitoric manifold over $P$, then we know every quasitoric manifold over a wedge of $P$.
    \item If one knows every toric manifold over $K$, then we know every toric manifold over a wedge of $K$.
    \end{itemize}
    We sometimes use the colloquial term ``toric objects'' to indicate one of the three categories. In this paper, we would claim that the above theorem is efficiently applicable to classify toric objects, and, hence, we can easily deduce the properties of given toric objects. In fact, we have many applications as below.

    To classify toric manifolds or topological toric manifolds, it seems natural to classify their underlying simplicial complexes first. Let $m$ be the number of rays of a complete non-singular fan of dimension $n$, and $K$ the corresponding simplicial complex of dimension $n-1$. The Picard number of $K$, denoted by $\Pic(K)$, is defined as $m-n$.\footnote{It coincides with the Picard number of the corresponding toric variety.} For instance, if $\Pic(K)=1$, then $K$ is the boundary complex of the $n$-simplex. It is known that only $\CP^n$ is the toric manifold supported by $K$. If $\Pic(K)=2$, then $K$ is the join of boundaries of two simplices (see \cite{Gru03}), and all toric manifolds over $K$ are classified by Kleinschmidt \cite{K}. More generally, every toric manifold over the join of boundaries of simplices is known as a generalized Bott manifold. Such manifolds are studied by several literatures such as \cite{oda88}, \cite{Ba}, \cite{Do}, \cite{CMS10}.

    However, not all simplicial spheres of $\Pic(K)=3$ support a toric manifold. Due to \cite{GKB90}, we have the complete criterion of simplicial complexes to support a toric manifold, and using this, Batyrev \cite{Ba} classified toric manifolds with Picard number $3$ as varieties. In this paper, we observe that every simplicial complex supporting smooth toric varieties is obtainable by a sequence of wedge operations from either a cross polytope or a pentagon (recall that the $n$-cross polytope is the dual of the $n$-cube). Hence, as an application of Theorem~\ref{thm:mainthm}, we classify toric manifolds with Picard number $3$ up to Davis-Januszkiewicz equivalence as quasitoric manifolds. Then, using the symmetry of a pentagon, we can get also the classification as varieties which is a reproving of the Batyrev's result.

    In the category of projective toric varieties, the situation does not go very well like Theorem~\ref{thm:mainthm}. In fact, there is a (singular) non-projective toric variety over $\wed_v(K)$ whose projections with respect to $v_1$ and $v_2$ are projective respectively. But we can still show projectivity of some families of toric varieties, containing toric manifolds with Picard number 3, with Shephard's projectivity criterion \cite{She71}, \cite{Ewa86}.
    The fact that every toric manifold of $ \operatorname{Pic}(K) \le 3$ is projective was originally shown by \cite{KS91}, but their method was lengthy and cumbersome case-by-case approach and the paper does not contain the whole proof due to its length and repetitive calculations. In Section~\ref{sec:applications}, a new and complete proof of the fact will be given. Moreover, we will provide a criterion of compact (singular) toric variety over  the join of boundaries of simplices with arbitrary Picard number to be projective. We note that when such a toric variety is smooth, it is known as a generalized Bott manifold which is always projective.

    When $M$ is a toric variety of complex dimension $n$, there is a canonical involution on $M$ and its fixed points form a real subvariety of real dimension $n$, called a \emph{real toric variety}. Similarly, there are ``real'' versions of topological toric manifolds and quasitoric manifolds called \emph{real topological toric manifolds} and \emph{small covers}, respectively. Such real analogues of toric objects also can be described as a $\Z_2$-version of characteristic map $(K,\lambda)$, that is, the codomain of $\lambda$ is $\Z_2^n$ instead of $\Z^n$. The map $\proj_v(\lambda)$ of a characteristic map over $\Z_2$ also can be defined similarly. Then, we have the $\Z_2$-version of Theorem~\ref{thm:mainthm} as the following:
    \begin{theorem}\label{thm:mainthm2}
        Let $K$ be a fan-like simplicial sphere and $v$ a given vertex of $K$. Let $(\wed_v(K), \lambda)$ be a characteristic map over $\Z_2$ and let $v_1$ and $v_2$ be the two new vertices of $\wed_v(K)$ created from the wedging. Then $\lambda$ is uniquely determined by the projections $\proj_{v_1}\lambda$ and $\proj_{v_2}\lambda$. Furthermore, $\lambda$ is non-singular if and only if  so are $\proj_{v_1}\lambda$ and $\proj_{v_2}\lambda$.
    \end{theorem}

    As corollaries, we classify and enumerate real topological toric manifolds and smooth real toric varieties with Picard number $3$. By \cite{CMS10}, every real topological toric manifold over  the join of boundaries of simplices are indeed a real toric variety known as a generalized real Bott manifold. Moreover, the classification of real toric manifolds over the join of boundaries of simplices is given in \cite{CMS10}, and the number of generalized real Bott manifolds with Picard number $3$ is presented in \cite{Choi08}.
    In this paper, if $K$ with $\Pic(K)=3$ supports a real topological toric manifold, then $K$ should be obtainable by a sequence of wedge operations from a 3-cross polytope, a pentagon, or a $4$-cyclic polytope with $7$ vertices. Equivalently, a Gale diagram of $K$ is a triangle, a pentagon, or a heptagon.
    Furthermore, we will give a complete classification of them up to Davis-Januszkiewicz equivalence, and count them. In particular, when $P$ is a simple polytope whose Gale diagram is a pentagon with assigned numbers $(a_1,a_3,a_5,a_2,a_4)$, every real topological toric manifold over $P$ is a real toric variety, and  the number \#DJ of them up to Davis-Januszkiewicz equivalence is
        \[
             \hbox{\#DJ}= 2^{a_1+a_4-1}+2^{a_2+a_5-1}+2^{a_3+a_1-1}+2^{a_4+a_2-1}+2^{a_5+a_3-1}-5.
         \]
    When $P$ is a polytope whose Gale diagram is a heptagon with arbitrary assigned numbers, no real topological toric manifold over $P$ is a real toric variety, and we have $\hbox{\#DJ}=2$. Meanwhile, although such a manifold is not a real toric variety, we can see that any characteristic map $(\partial P^\ast ,\lambda)$ over $\Z_2$ is congruent to some characteristic map $(\partial P^\ast , \widetilde{\lambda})$ over $\Z$ up to modulo $2$. This observation provides an affirmative partial answer to so-called the \emph{lifting problem} which asks whether for given $K$, any real topological toric manifold over $K$ can be realized as fixed points of the conjugation of a topological toric manifold or not. That is, the answer to the lifting problem is affirmative for $\Pic(K) \leq 3$.

    The paper is organized as follows. In Section~\ref{sec:wedge}, we define wedge operations of simplicial complexes and study some of their properties related to toric objects. In Section~\ref{sec:torictopology}, we introduce some categories containing toric objects and their associated combinatorial objects such as fans, multi-fans, and characteristic maps. In Section~\ref{sec:toricandwedge} we prove the main result. In Section~\ref{sec:wedgeandprojectivity}, we introduce the Shephard diagram and Shephard's criterion of projectivity of toric varieties. In Section~\ref{sec:applications}, as an application of the main result, we give a classification of smooth toric varieties of Picard number 3. In Section~\ref{sec:projectivity}, we prove that smooth toric varieties of Picard number 3 are projective and give a criterion of when a toric variety over the join of boundaries of simplices is projective. We classify and count real topological toric manifolds over $K$ with $\Pic(K)=3$  in Section~\ref{sec:smallcover}.
    Lastly, we introduce the lifting problem of topological toric manifolds over $K$ and prove it for $\Pic(K) \le 3$ in Section~\ref{sec:liftingproblem}.

\section{Wedge operations of simplicial complexes}\label{sec:wedge}
    A \emph{simplicial complex} $K$ on a finite set $V$ is a collection of subsets of $V$ satisfying
    \begin{enumerate}
      \item if $v \in V$, then $\{ v \} \in K$,
      \item if $\sigma \in K$ and $\tau \subset \sigma$, then $\tau \in K$.
    \end{enumerate}

    Each element $\sigma \in K$ is called a \emph{face} of $K$. The dimension of $\sigma$ is defined by $\dim(\sigma)=|\sigma|-1$. The \emph{dimension} of $K$ is defined by $\dim(K) = \max \{ \dim(\sigma)\mid \sigma\in K\}$.

    There is a useful way to construct new simplicial complexes from a given simplicial complex introduced in \cite{BBCG10}. We briefly present the construction here. Let $K$ be a simplicial complex of dimension $n-1$ on vertices $V=[m] = \{1,2, \ldots, m\}$. A subset $\tau \subset V$ is called a \emph{non-face} of $K$ if it is not a face of $K$. A non-face $\tau$ is \emph{minimal} if any proper subset of $\tau$ is a face of $K$. Note that a simplicial complex is determined by its minimal non-faces.

    In the setting above, let $J=(j_1, \ldots, j_m)$ be a vector of positive integers. Denote by $K(J)$ the simplicial complex on vertices
    \[
        \{ \underbrace{1_1,1_2,\ldots,1_{j_1}},\underbrace{{2_1},2_2,\ldots,{2_{j_2}}},\ldots, \underbrace{{m_1},\ldots,{m_{j_m}}} \}
    \]
    with minimal non-faces
    \[
        \{ \underbrace{{(i_1)_1},\ldots,{(i_1)_{j_{i_1}}}},\underbrace{{(i_2)_1},\ldots,{(i_2)_{j_{i_2}}}},\ldots, \underbrace{{(i_k)_1},\ldots,{(i_k)_{j_{i_k}}}} \}
    \]
    for each minimal non-face $\{{i_1},\ldots,{i_k}\}$ of $K$.

    There is another way to construct $K(J)$ called the \emph{simplicial wedge construction}. Recall that for a face $\sigma$ of a simplicial complex $K$, the \emph{link} of $\sigma$ in $K$ is the subcomplex
    \[
         \link_K\sigma := \{ \tau \in K \mid \sigma\cup\tau\in K,\;\sigma\cap\tau=\varnothing\}
    \]
    and the \emph{join} of two disjoint simplicial complexes $K_1$ and $K_2$ is defined by
    \[
        K_1 \star K_2 = \{ \sigma_1 \cup \sigma_2 \mid \sigma_1 \in K_1,\; \sigma_2 \in K_2\}.
    \]
    Let $K$ be a simplicial complex with vertex set $[m]$ and fix a vertex $i$ in $K$. Consider a 1-simplex $I$ whose vertices are ${i_1}$ and ${i_2}$ and denote by $\partial I = \{i_1,\,i_2\}$ the 0-skeleton of $I$. Now, let us define a new simplicial complex on $m+1$ vertices, called the \emph{(simplicial) wedge} of $K$ at $i$, denoted by $\wed_{i}(K)$, by
    \[ \wed_{i}(K)= (I \star \link_K\{i\}) \cup (\partial I \star (K\setminus\{i\})), \]
    where $K\setminus\{i\}$ is the induced subcomplex with $m-1$ vertices except $i$. The operation itself is called the \emph{simplicial wedge operation} or the \emph{(simplicial) wedging}. See Figure~\ref{fig:wedge}.


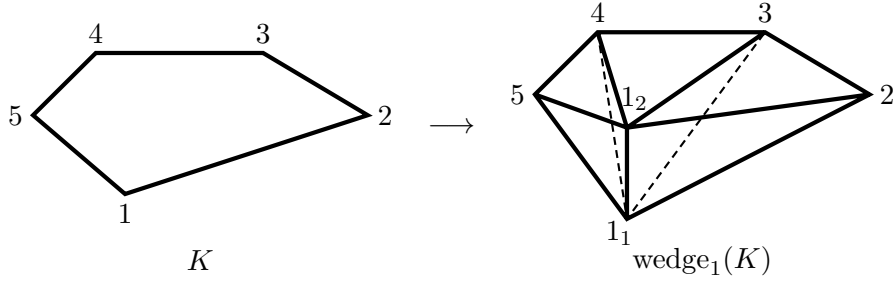
\begin{figure}[h]
    \begin{tikzpicture}[scale=.55]
        \coordinate [label=below:$1$](11) at (-9.8,0.6);
        \coordinate [label=right:$2$](22) at (-4,2.5);
        \coordinate [label=above:$3$](33) at (-6.5,4);
        \coordinate [label=above:$4$](44) at (-10.5,4);
        \coordinate [label=left:$5$](55) at (-12,2.5);
        \draw (-2,2.2) node {$\longrightarrow$};
        \coordinate [label={[xshift=-2.8pt,yshift=2.8pt]below:$1_1$}](1_1) at (2.2,0);
        \coordinate [label={[xshift=2.8pt,yshift=2.8pt]above:$1_2$}](1_2) at (2.2,2.2);
        \coordinate [label=right:$2$](2) at (8,3);
        \coordinate [label=above:$3$](3) at (5.5,4.5);
        \coordinate [label=above:$4$](4) at (1.5,4.5);
        \coordinate [label=left:$5$](5) at (0,3);
        \draw (-8,-1) node {$K$}
            (4,-1) node{$\wed_1(K)$};
        \draw [ultra thick] (11)--(22)--(33)--(44)--(55)--cycle
            (1_1)--(2)--(3)--(4)--(5)--cycle
            (1_2)--(1_1)
            (1_2)--(2)
            (1_2)--(3)
            (1_2)--(4)
            (1_2)--(5);
        \draw [thick, densely dashed] (4)--(1_1)--(3);
    \end{tikzpicture}
    \caption{Illustration of a wedge of $K$}\label{fig:wedge}
\end{figure}

     It is an easy observation to show that $\wed_{i}(K)=K(J)$ where $J=(1,\ldots,1,2,1,\ldots,1)$ is the $m$-tuple with 2 as the $i$-th entry. By consecutive application of this construction starting from $J=(1,\ldots,1)$, we can produce $K(J)$ for any $J$. Although there is some ambiguity to proceed from $J=(j_1,\ldots, j_m)$ to $J'=(j_1,\ldots,j_{i-1}, j_i+1,j_{i+1},\ldots,j_m)$ if $j_i\ge 2$, we have no problem since any choice of the vertex yields the same minimal non-faces of the resulting complex $\wed_v(K(J))=K(J')$ keeping in mind the original definition of $K(J)$. In conclusion, one can obtain a simplicial complex $K(J)$ by successive simplicial wedge constructions starting from $K$, independent of order of wedgings.

    Related to the simplicial wedging, we recall some hierarchy of simplicial complexes. Among simplicial complexes, simplicial spheres form a very important subclass.
    \begin{definition}\label{def:cpxhierarchy}
        Let $K$ be a simplicial complex of dimension $n-1$.
        \begin{enumerate}
            \item $K$ is called a \emph{simplicial sphere} of dimension $n-1$ if its geometric realization $|K|$ is homeomorphic to a sphere $S^{n-1}$.
            \item $K$ is called \emph{star-shaped in $p$} if there is an embedding of $|K|$ into $\R^{n}$ and a point $p\in \R^n$ such that any ray from $p$ intersects $|K|$ once and only once. 
                The geometric realization $|K|$ itself is also called star-shaped.
            \item $K$ is said to be \emph{polytopal} if there is an embedding of $|K|$ into $\R^{n}$ which is the boundary of a simplicial $n$-polytope $P^\ast$.
        \end{enumerate}
    \end{definition}
    We have a chain of inclusions
    \begin{multline*}
        \text{simplicial complexes}\supset\text{simplicial spheres} \\
        \supset\text{star-shaped complexes} \supset\text{polytopal complexes}.
    \end{multline*}

    It is worthwhile to observe that each category of simplicial complexes above is closed under the wedge operation as follows.

    \begin{proposition}\label{prop:propofwedge}
        Let $K$ be a simplicial complex and $v$ its vertex. Then the followings hold:
        \begin{enumerate}
            \item {If $K$ is a simplicial sphere, then so is $\wed_v(K)$. \label{stm:sphere}}
            \item $\wed_v(K)$ is star-shaped if and only if  so is  $K$. \label{stm:star}
            \item $\wed_v(K)$ is polytopal if and only if so is $K$. \label{stm:polytopal}
        \end{enumerate}
    \end{proposition}

    \begin{proof}
        To prove \eqref{stm:sphere}, we recall the definition of $\wed_v(K)$:
        \[
            \wed_v(K)= (I \star \link_K\{v\}) \cup (\partial I \star (K\setminus\{v\})).
        \]
        Observe that the join $\partial I\star K$, or the suspension of $K$, is
        \[
            \partial I\star K = (\partial I \star \overline{\operatorname{St}_K\{v\}}) \cup (\partial I \star (K\setminus\{v\}),
        \]
        where $\overline{\operatorname{St}_K\{v\}}$ means the closed star of $v$. Observe that
        \[
             \partial I \star \overline{\operatorname{St}_K\{v\}} = \partial I \star \{pt\} \star \link_K\{v\}
        \]
        is a subdivision of $I \star \link_K\{v\}$ and therefore the geometric realizations $|\wed_v(K)|$ and $|\partial I\star K|$ are homeomorphic. {But, the suspension of a sphere is again a sphere}, so it is done.

        Next, we are going to show \eqref{stm:star}. The `only if' part  follows from \cite[Section~2]{HM03} (or see Lemma~\ref{lem:projection}). For `if' part, suppose that $K$ is a star-shaped sphere of dimension $n-1$. In other words, there is an embedding $|K|$ into $\R^n$ such that the origin is in the kernel of $|K|$. Fix a vertex $v$ of $|K|$. We regard $\R^n$ as a hyperplane of $\R^{n+1}$ so that $\R^{n+1}=\{(p,x)\mid p\in\R^n,\;x\in\R\}$. Put $A:=\{(v,1),(v,-1)\}\subset\R^{n+1}$ and consider the geometric join
        \[
            A\star_G |K| := \{(1-t)a + tp \mid 0\le t\le 1,\; a\in A,\; p\in |K|\}.
        \]
        We claim that $A\star_G |K|$ is the wanted geometric realization of $\wed_v(K)$ which is star-shaped. The fact that $A\star_G |K|$ has the face structure same as $\wed_v(K)$ can be shown easily since the vertex $(v,0)$ is in the convex hull of $A$ --- recall that $\wed_v(K)$ is subdivided to $\partial I\star K$ by adding a vertex into the edge $I$.
        Therefore we remain to show that $A\star_G |K|$ is star-shaped in the origin $O$. Consider a ray $R$ starting from $O$. It must intersect $A\star_G |K|$ and thus we are enough to show the uniqueness of the intersection. Suppose
        \[
            (1-t)a + tp \in R\cap (A\star_G |K|)
        \]
        for some $a\in A$, $p\in |K|$, and $0\le t\le 1$. The ray $R$ is given by $R = \{s(1-t)a + stp\mid s\ge 0 \}.$
        Suppose that
        \[
            s(1-t)a + stp = (1-t')b + t'p'
        \]
        for some $0\le t'\le 1$, $b\in A$, and $p'\in |K|$. There are two cases:
        \begin{itemize}
            \item $t=1$. Since $\R^{n+1} = \langle a\rangle\oplus \R^n$, one has $t'=1$ and $sp=p'$. The star-shapedness of $|K|$ implies $s=1$.
            \item $t\ne 1$. Similarly, one has $t'\ne 1$ and $b=a$. A property of direct sums says that $s(1-t) = 1-t'$ and $stp = t'p'$. Hence $p=p'$, $t=t'$ and $s=s'$.
        \end{itemize}
        In conclusion, we have proven that $R$ intersects $A\star_G |K|$ exactly once.

        The proof of  \eqref{stm:polytopal} could be done similarly to \eqref{stm:star}, but we introduce a different approach depicted below.
    \end{proof}

    When $K$ is polytopal, we often regard $K$ as the boundary complex of a simple polytope $P$. To be more precise, let $K$ be the boundary of a simplicial polytope $Q$. Then the dual polytope to $Q$ is a simple polytope $P$. Recall that an $n$-dimensional polytope $P$ is called \emph{simple} if exactly $n$ facets (or codimension 1 faces) intersect at each vertex of $P$. We follow \cite{KW67} to define the notion of the (polytopal) wedge.
    Let $P \subseteq \R^n$ be a polytope of dimension $n$ and $F$ a face of $P$. Consider a polyhedron $P\times [0,\infty)\subseteq \R^{n+1}$ and identify $P$ with $P\times \{0\}$. Pick a hyperplane $H$ in $\R^{n+1}$ so that $H\cap P = F$ and $H$ intersects the interior of $P\times [0,\infty)$. Then $H$ cuts $P\times [0,\infty)$ into two parts. The part which contains $P$ is an $(n+1)$-polytope and it is combinatorially determined by $P$ and $F$, and it is called the \emph{(polytopal) wedge of $P$ at $F$} and denoted by $\wed_F(P)$. Note that $\wed_F(P)$ is simple if $P$ is simple and $F$ is a facet of $P$. See Figure~\ref{fig:wedge_simple}.


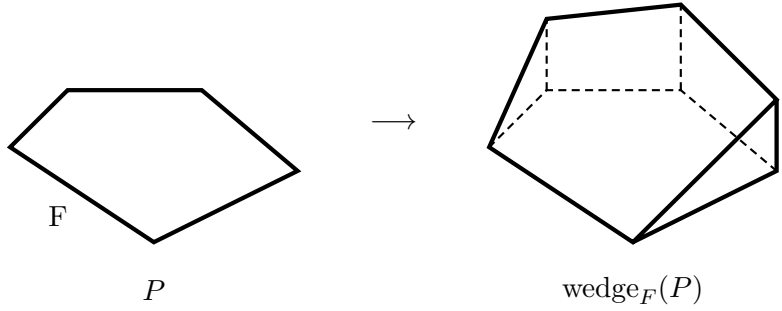
\begin{figure}[h]
    \begin{tikzpicture}[scale=.63]
        \draw [ultra thick]
            (-7,0)--(-4,1.5)--(-6,3.2 )--(-8.8,3.2 )--(-10,2)--cycle
            (3,0)--(6,3.0)--(4,5.0 )--(1.2,4.7 )--(0,2)--cycle
            (3,0)--(6,1.5)--(6,3.0);
        \draw [thick, densely dashed]
            (3,0)--(6,1.5)--(4,3.2 )--(1.2,3.2 )--(0,2)--cycle
            (4,3.2)--(4,5.0)
            (1.2,3.2)--(1.2,4.7);
        \draw (-7,-1) node {$P$}
            (3,-1) node {$\wed_{F}(P)$}
            (-9,0.55) node {F}
            (-2,2.5) node {$\longrightarrow$};
    \end{tikzpicture}
    \caption{Illustration of a wedge of $P$}\label{fig:wedge_simple}
\end{figure}

    The next lemma is due to \cite{PB80}.

    \begin{lemma}\label{lem:wedge}
        Assume that $P$ is a simple polytope and $F$ is a facet of $P$. Then the boundary complex of $\wed_F(P)$ is the same as the simplicial wedge of the boundary complex of $P$ at $F$.
    \end{lemma}
    \begin{proof}
        Let $\cF=\{F_1,\ldots,F_m\}$ be the set of facets of $P$ and put $F=F_1$. Note that $\wed_F(P)$ has $m+1$ facets. Two of them contain $F\times \{0\}$ and we label them as $F'_{1_1}$ and $F'_{1_2}$, respectively. Each facet except $F'_{1_1}$ and $F'_{1_2}$ is a subset of $F_i\times[0,\infty)$ for some $1<i\le m$, which is labeled as $F'_i$. Thus the new facet set is $\cF'=\{F'_{1_1},F'_{1_2},F'_2,\ldots, F'_m\}$. The boundary complex of $\wed_F(P)$ is a simplicial complex whose vertex set is $\{1_1,1_2,2,\ldots,m\}$. Now consider a set $I\subseteq\{1_1,1_2,2,\ldots,m\}$ and check whether the intersection $\bigcap_{i\in I}F'_i\subset \wed_F(P)$ is nonempty. There are three cases:
        \begin{itemize}
            \item Case I: $I$ contains neither $1_1$ nor $1_2$. In this case, $\bigcap_{i\in I}F'_i$ is nonempty in $\wed_F(P)$ if and only if $\bigcap_{i\in I}F_i$ is nonempty in $P$.
            \item Case II: $I$ contains both $1_1$ and $1_2$. Since $F'_{1_1}\cap F'_{1_2} = F_1\times\{0\}$, $\bigcap_{i\in I}F'_i$ is nonempty if and only if $\bigcap_{i\in I\setminus\{1_1,1_2\}}F_i$ is nonempty. It is equivalent to that $I\setminus\{1_1,1_2\}$ is a subset of the link of $1$ in the boundary complex of $P$.
            \item Case III: $I$ contains exactly one of ${1_1}$ or ${1_2}$. Assume that ${1_1}\in I$. In this case, $\bigcap_{i\in I}F'_i$ is nonempty in $\wed_F(P)$ if and only if $\bigcap_{i\in I\setminus \{1_1\}}F_i$ is nonempty in $P$. Therefore, this case coincides with $\partial I \star (K\setminus\{F_1\})$ where $K$ is the boundary complex of $P$.
        \end{itemize}
        The proof is completed putting all cases together. Note that if the intersection is a vertex, then it is Case II or case III, since every vertex of $\wed_F(P)$ is in $F'_{1_1}$ or $F'_{1_2}$.
    \end{proof}
    Suppose $P$ is a simple polytope and $\cF=\{F_1,\ldots,F_m\}$ is the set of facets of $P$. Let $J=(a_1,\ldots,a_m)\in\N^m$ be a vector of positive integers. Then define $P(J)$ by the combinatorial polytope obtained by consecutive polytopal wedgings analogous to the construction of $K(J)$ with simplicial wedgings. Lemma~\ref{lem:wedge} guarantees that if the boundary complex of $P$ is $K$, then the boundary complex of $P(J)$ is $K(J)$.

    {
    \begin{remark}
        The converse of \eqref{stm:sphere} in Proposition~\ref{prop:propofwedge} does not hold because of the famous Double Suspension Theorem of Edwards and Cannon \cite{Ca79} which states that every double suspension $SSM$ of a homology $n$-sphere $M$ is homeomorphic to an $(n+2)$-sphere.
    \end{remark}
    }

\section{Toric topology and combinatorial objects}\label{sec:torictopology}

    A fundamental result of toric geometry is that there is a bijection between toric varieties of dimension $n$ and (rational) fans of real dimension $n$. One could regard a fan as a ``combinatorial'' object associated to a toric variety. Similarly, objects in toric topology such as topological toric manifolds, quasitoric manifolds, and torus manifolds, have their associated combinatorial objects respectively. In this section, we briefly introduce them focusing on combinatorial objects.

    \subsection{Toric varieties and fans}

    Let us review the definition of a fan. For a subset $X \subset \R^n$, the \emph{positive hull} of $X$, denoted by $\pos X$, is the set of positive linear combinations of $X$, that is,
    \[
        \pos X = \left\{ \sum_{i=1}^{k}a_ix_i\mid a_i\ge 0,\:x_i\in X\right\}.
    \]

    By convention, we put $\pos X = \{0\}$ if $X$ is empty. A subset $C$ of $\R^n$ is called a \emph{polyhedral cone}, or simply a \emph{cone}, if there is a finite set $X$ of vectors, called the \emph{set of generators} of the cone, such that $C = \pos X$. The elements of $X$ is called \emph{generators} of $C$. We also say that $X$ positively spans the cone $C$. A subset $D$ of $C$ is called a \emph{face} of $C$ if there is a hyperplane $H$ such that $C\cap H = D$ and $C$ does not lie in both sides of $H$. A cone is by convention a face of itself and all other faces are called \emph{proper}.

    A cone is called \emph{strongly convex} if it does not contain a nontrivial linear subspace. In this paper, every cone is assumed to be strongly convex. A polyhedral cone is called \emph{simplicial} if its generators are linearly independent, and \emph{rational} if every generator is in $\Z^n$. A rational cone is called \emph{non-singular} if its generators are unimodular, i.e., they are a part of an integral basis of $\Z^n$.

    A \emph{fan} $\Sigma$ of real dimension $n$ is a set of cones in $\R^n$ such that
    \begin{enumerate}
        \item if $C \in \Sigma$ and $D$ is a face of $C$, then $D \in\Sigma$,
        \item and for $C_1,\,C_2\in \Sigma$, $C_1 \cap C_2$ is a face of $C_1$ and $C_2$ respectively.
    \end{enumerate}
      A fan $\Sigma$ is said to be \emph{rational} (resp. \emph{simplicial} or \emph{non-singular}) if every cone in $\Sigma$ is rational (resp. simplicial or non-singular). Remark that the term ``fan'' is used for rational fans in most literature, especially among toric geometers. We will sometimes use the term ``real fan'' to emphasize that generators need not be integral vectors.

    If a fan $\Sigma$ is simplicial, then we can think of a simplicial complex $K$, called the \emph{underlying simplicial complex} of $\Sigma$, whose vertices are generators of cones of $\Sigma$ and whose faces are the sets of generators of cones in $\Sigma$ (including the empty set). We also say that $\Sigma$ is a fan over $K$. In this paper a fan is assumed to be simplicial {unless otherwise mentioned}.

    A fan $\Sigma$ is called \emph{complete} if the union of cones in $\Sigma$ covers all of $\R^n$. Observe that the underlying simplicial complex of a fan is a simplicial sphere if and only if the fan is complete. It is a well-known fact that a rational fan is complete (resp. non-singular) if and only if its corresponding toric variety is compact (resp. smooth). A compact smooth toric variety is called a \emph{toric manifold} in this paper. We remark that a toric variety is an orbifold if and only if its corresponding fan is simplicial.

    We close this subsection giving definition of two notions relating a fan to a polytope. A fan is said to be \emph{weakly polytopal} if its underlying simplicial complex is polytopal in the sense of Definition~\ref{def:cpxhierarchy}. A fan $\Sigma$ is called \emph{strongly polytopal} if there is a simplicial polytope $P^*$, called a \emph{spanning polytope}, such that $0\in\operatorname{int}P^*$ and
    \[
        \Sigma = \{\pos \sigma \mid \sigma\text{ is a proper face of $P^*$}\}.
    \]
    Observe that the underlying complex of $\Sigma$ is $\partial P^*$. Therefore strong polytopalness implies weak polytopalness.

    It is a well-known fact from convex geometry that a fan $\Sigma$ is strongly polytopal if and only if $\Sigma$ is the normal fan of a simple polytope $P$. For a given simple $n$-polytope $P\subset \R^n$, correspond to each facet $F$ the outward normal vector $N(F)$. The \emph{normal fan} $\Sigma$ of $P$ is a collection of cones
    \[
        \Sigma = \Bigl\{\pos \{N(F)\mid F\supset f\} \;\Bigm|\; f\text{ is a proper face of $P$}\Bigr\}.
    \]
    \begin{theorem}\cite[Theorem~V.4.4]{Ewa96}
        A fan $\Sigma$ is strongly polytopal whose spanning polytope is a simplicial polytope $P^*$ if and only if it is a normal fan of the dual polytope $P^{**}=P$.
    \end{theorem}

    The toric variety corresponding to a (rational) strongly polytopal fan is known to be a projective algebraic variety, that is, a subvariety of a complex projective space. We call such a variety a \emph{projective toric variety}. If it is non-singular, then we call it a \emph{projective toric manifold}. We say that a simple polytope is a \emph{Delzant polytope} if its normal fan is non-singular. By definition, a simple $n$-polytope $P\subset \R^n$ is Delzant if and only if outward (integral and primitive) vectors of $P$ normal to facets form an integral basis of $\Z^n$ at each vertex of $P$. We have fundamental bijections
    \begin{multline*}
        \text{strongly polytopal non-singular fans}\Longleftrightarrow \text{Delzant polytopes} \\
        \Longleftrightarrow\text{projective toric manifolds}.
    \end{multline*}

    \subsection{Topological toric manifolds and characteristic maps}

    It is easy to see that every information of a rational fan $\Sigma$ of dimension $n$ can be recovered from the underlying simplicial complex $K=K(\Sigma)$ of dimension $\le n-1$ and a map $\lambda=\lambda(\Sigma)\colon V(K)\to \Z^n$, where $V(K)$ denotes the vertex set of $K$ and $\lambda$ maps a vertex of $K$ to the primitive integral vector corresponding to the 1-cone of $\Sigma$. We call the pair $(K(\Sigma),\lambda(\Sigma))$ the \emph{characteristic map} of $\Sigma$. More generally, we give the following definition of characteristic maps. Before that, recall that a simplicial sphere $K$ is said to be \emph{fan-like} if there is a complete fan over $K$. Observe that fan-likeness and star-shapedness of Definition~\ref{def:cpxhierarchy} are equivalent properties of simplicial spheres. (This is not true in general polyhedral spheres; see \cite{Ewa96} for an example.)

    \begin{definition}
        A \emph{characteristic map} of dimension $n$ (over $\Z$) is defined as the pair $(K,\lambda)$ of an abstract simplicial complex $K$ of dimension $\le n-1$ and a map $\lambda\colon V(K)\to \Z^n$ so that $\{\lambda(i)\mid i\in I\}$ is a linearly independent set over $\R$ for any face $I$ of $K$, where $V(K) =[m]$ is the vertex set of $K$. Moreover, we define the following:
        \begin{enumerate}
            \item {$(K,\lambda)$ is said to be \emph{primitive} if each $\lambda(i)$ is a primitive vector.}
            \item $(K,\lambda)$ is called \emph{complete} if $K$ is a fan-like simplicial sphere of dimension $n-1$.
            \item $(K,\lambda)$ is called \emph{non-singular} if for every face of $I\in K$, the set of vectors $\{\lambda(i)\mid i\in I\}$ positively spans a non-singular cone.
            \item {$(K,\lambda)$ is called \emph{fan-giving} if the set of cones
            \[
                \Bigl\{\pos\{\lambda(i)\mid i\in \sigma\}\Bigm| \sigma\in K\Bigr\}
            \]
            is a fan.}
        \end{enumerate}
        Sometimes we call the map $\lambda$ itself a characteristic map. {Note that every non-singular characteristic map is primitive. In addition, a characteristic map is primitive and fan-giving if and only if there is a fan $\Sigma$ such that $(K,\lambda)=(K(\Sigma),\lambda(\Sigma))$.}
    \end{definition}
%

    One could think of the ``real'' version of a characteristic map so that the mapped vectors are in $\R^n$, not necessarily $\Z^n$. We still can define completeness and fan-givingness in that case.
    By definition, if a characteristic map $(K,\lambda)$ of dimension $n$ is complete and non-singular, then for any maximal face $I$ of $K$, the vectors $\{\lambda(i)\mid i\in I\}$ form an integral basis of $\Z^n$. Note that not every characteristic map defines a fan because overlapping cones may exist. We remark that the term ``characteristic function'' is used in some literatures including \cite{DJ91}, in the meaning of ``complete non-singular characteristic map on the boundary complex of a simple polytope $P$''. Note that a fan-giving characteristic map corresponds to a complete (resp. non-singular) fan if and only if the characteristic map is complete (resp. non-singular).

    A topological toric manifold introduced in \cite{IFM12} is a topological generalization of a toric manifold. A closed smooth manifold $M$ of dimension $2n$ with an effective smooth action of $(\C^*)^n$ is a \emph{topological toric manifold} if one of its orbits is open and dense and $M$ is covered by finitely many invariant open subsets each of which is equivariantly diffeomorphic to a smooth representation space of $(\C^*)^n$. Note that toric manifolds are topological toric manifolds by definition. Let us briefly recall its combinatorial counterpart called a topological fan.

    \begin{definition}
        Let $K$ be a simplicial complex of dimension $\le n-1$ and let $\mathbf{z}\colon V(K)\to \C^n$ and $\lambda\colon V(K) \to \Z^n$ be maps. The triple $\Sigma=(K,\mathbf{z},\lambda)$ is called a \emph{topological fan} of dimension $n$ if the followings hold:
        \begin{enumerate}
            \item $(K,\operatorname{Re}\mathbf{z})$ is a fan-giving real characteristic map, where $\operatorname{Re}\mathbf{z}\colon V(K)\to \R^n$ denotes the coordinate-wise real part of the vector $\mathbf{z}$ and
            \item $(K,\lambda)$ is a {primitive} characteristic map.
        \end{enumerate}
        We say $\Sigma$ is \emph{complete} if $(K,\operatorname{Re}\mathbf{z})$ is complete and it is \emph{non-singular} if $(K,\lambda)$ is non-singular.
    \end{definition}
    Let $M$ be a topological toric manifold of dimension $2n$. We say that a closed connected smooth submanifold of $M$ of real codimension two is a \emph{characteristic submanifold} if it is fixed pointwise under some $\C^\ast$-subgroup of $(\C^\ast)^n$. According to \cite{IFM12}, there are only finitely many characteristic submanifolds. A choice of an orientation on each characteristic submanifold together with the orientation of $M$ is called an \emph{omniorientation} on $M$.
    By the classification theorem of \cite{IFM12}, every omnioriented $2n$-dimensional topological toric manifold bijectively corresponds to a complete non-singular topological fan of dimension $n$.    Note that if $\mathbf{z}(i)=\lambda(i)$ for all $i\in V(K)$, then the topological fan corresponds to an ordinary fan. Furthermore, in this case, the notion of completeness and non-singularity above generalizes that for an ordinary fan.

    We denote by $T^n=(S^1)^n$ the compact torus with dimension $n$. Recall that a topological toric manifold $M$ is equipped by a $(\C^*)^n$-action. Thus $M$ has a natural $T^n$-action as a subgroup of $(\C^*)^n$. As a $T^n$-manifold, $M$ can be characterized by the following theorem of \cite{IFM12}.

    \begin{theorem}\cite[Theorem~7.2]{IFM12}
        Let $\Sigma=(K,\mathbf{z},\lambda)$ be a complete non-singular topological fan of dimension $n$ and $M$ be the corresponding topological toric manifold. Then the $T^n$-equivariant homeomorphism type of $M$ is independent of $\mathbf{z}$.
    \end{theorem}
    In other words, any omnioriented topological toric manifold as a $T^n$-manifold is determined by a complete non-singular characteristic map $(K,\lambda)$ equipped with an orientation of $|K|$. We say that $(K,\lambda)$ is an \emph{oriented characteristic map} if an orientation of $|K|$ is fixed and the orientation is called an \emph{orientation} of $(K,\lambda)$. In this paper, we regard topological toric manifolds as $T^n$-manifolds. In addition, even if we do not have defined, we will use the notation $\lambda(M)$ for the characteristic map corresponding to a topological toric manifold $M$, and $M(\lambda)$ or $M(K,\lambda)$ vice versa. Similarly, $K(M)$ means the underlying complex for $M$. The same goes for $\Sigma(M)$ and $M(\Sigma)$ for a toric manifold $M$ and its fan $\Sigma$.

    \subsection{Quasitoric manifolds}

    Quasitoric manifolds, introduced in \cite{DJ91}, are another topological analogue of toric manifolds. A closed smooth manifold $M$ of dimension $2n$ with a smooth action of $T^n$ is called a \emph{quasitoric manifold} over a simple polytope $P$ if
    \begin{enumerate}
        \item the action of $T^n$ on $M$ is locally standard and
        \item the orbit space $M/T^n$ is $P$.
    \end{enumerate}
    Its combinatorial object is called the characteristic function. Let $\cF$ be the set of facets of a simple $n$-polytope $P$. A \emph{characteristic function} over $P$ is a map $f\colon \cF\to \Z^n$ satisfying non-singularity condition
    \begin{multline*}
        F_{i_1},\ldots,F_{i_n}\in\cF\text{ intersect at a vertex of }P \\
        \Longrightarrow \{f(F_{i_j})\mid j=1,\ldots,n\}\text{ is an integral basis of }\Z^n.
    \end{multline*}
    Since the dual of $P$ is a simplicial $n$-polytope, $f$ induces a complete non-singular characteristic map $(K,\lambda)$ whose underlying complex is a polytopal sphere. Polytopal spheres are star-shaped by convexity and therefore $(K,\lambda)$ defines a topological toric manifold $M'$. Indeed, $M'$ is equivariantly homeomorphic to $M$ (for a proof see Section 10 of \cite{IFM12}). This shows any quasitoric manifold is a topological toric manifold. Conversely, there is a topological toric manifold whose underlying complex is not polytopal (see \cite{IFM12} for an example and note that in the example the underlying complex is the Barnette sphere).

    There are many examples of quasitoric manifolds which are not toric manifolds. We remark that it is an open problem whether there exists a toric manifold which is not a quasitoric manifold or not. Note that this is about finding a complete non-singular fan which is not weakly polytopal.

    \subsection{Characteristic maps and Todd genera}\label{subsec:charandmultifans}

    A closed connected smooth orientable manifold $M$ of dimension $2n$ is called a \emph{torus manifold} if it is equipped with an effective $T^n$-action which has a nonempty fixed point set. Torus manifolds make a large class of manifolds properly containing topological toric manifolds (and, obviously, toric manifolds and quasitoric manifolds). A torus manifold has its own combinatorial object, called a multi-fan, which can be roughly understood as a collection of cones similar to a fan but the cones may ``overlap''. 
    Although we do not present the precise definition of multi-fans, we use the concept of overlapping cones to consider fan-givingness of a characteristic map. For further reading for multi-fans, refer to \cite{HM03}.

    Let $(K,\lambda)$ be a characteristic map of dimension $n$ and $I\in K$ a face of $K$. One defines the cone over $I$ be the positive hull $\pos \{\lambda(i)\mid i\in I\}$ and denote it by $\angle\lambda_I$. From now on, we assume $(K,\lambda)$ is complete and follow an argument of Section 4 of \cite{IM12}. First, we consider the (geometric) simplicial complex $|K|$ which is an $(n-1)$-dimensional sphere. We set

    \[
        \sigma_I :=\left\{\sum_{i\in I}a_i\mathbf{e}_i\;\middle|\; \sum_{i\in I}a_i = 1,\ a_i\ge 0\right\}\subset\R^m\quad\text{ for }I\in K,
    \]
    where $\mathbf{e}_i$ is the $i$-th coordinate vector of $\R^m$. The geometric realization $|K|$ of $K$ is given by
    \[
        |K| = \bigcup_{i\in K}\sigma_I.
    \]
    Let us write $\lambda(i)=\lambda_i\in\Z^n$ for $i=1,\dotsc,m$. We define a map $f_\lambda\colon |K|\to S^{n-1}$ by
    \[
        f_\lambda|_{\sigma_I} \left(\sum_{i\in I}a_i\mathbf{e}_i\right) = \frac{\sum_{i\in I}a_i\lambda_i}{|\sum_{i\in I}a_i\lambda_i|}.
    \]
    Observe that $f_\lambda$ is a homeomorphism if and only if $\lambda$ is fan-giving.

    Fix an orientation of $|K|$. For each cone $\angle\lambda_I$ of dimension $n$ (or an $(n-1)$-face $I$ of $K$ equivalently), we assign $+1$ or $-1$, called the \emph{weight function} $w(I)$ of $\lambda$, so that
    \[
        w(I)=\left\{
             \begin{array}{ll}
               +1, & \hbox{The orientations of $\sigma_I\subset|K|$ and $f_\lambda|_{\sigma_I}$ coincide;} \\
               -1, & \hbox{otherwise.}
             \end{array}
           \right.
    \]
    \begin{definition}\label{def:positivechr}
    An oriented complete characteristic map $\lambda$ is said to be \emph{positive} if every $w(I)$ is positive. In this case the orientation of $|K|$ is called the \emph{positive orientation} of $\lambda$ and $\lambda$ is said to be \emph{positively oriented} or shortly \emph{positive}.
    \end{definition}

    The next definition is confirmed available by Theorem~4.2 of \cite{Mas99}.
    \begin{definition}
        Let $(K,\lambda)$ be a given complete characteristic map of dimension $n$. Let $v$ be a generic vector in $\R^n$ in the sense that $v$ does not lie in any cone $\angle \lambda_I$ for a non-maximal face $I$ of $K$. Then the value
        \[
            \sum_{I\colon v\in \angle\lambda_I}w(I)
        \]
        is independent of the choice of $v$. We call the value the \emph{Todd genus} of $(K,\lambda)$ and denote it by $\operatorname{Todd}(\lambda)$.
    \end{definition}
    This definition actually says that the Todd genus $\operatorname{Todd}(\lambda)$ is the degree of $f_\lambda$ as a map between spheres. We can check whether a characteristic map is fan-giving. The following is a version of Lemma~4.1 of \cite{IM12}.
    \begin{proposition}\label{prop:toddandfangiving}
        A complete characteristic map $\lambda$ is fan-giving if and only if $\lambda$ is positive and  $\operatorname{Todd}(\lambda)=1$.
    \end{proposition}
    \begin{proof}
        The only-if part is obvious. Suppose that $\lambda$ is fan-giving. Since $\lambda$ is positive, the weight $w(I)$ is nonnegative for all $I$ such that $v\in \angle\lambda_I$. Because the sum of such $w(I)$ is one, any generic vector $v$ is contained exactly one maximal cone $\angle\lambda_I$. Hence $\lambda$ is fan-like.
    \end{proof}

%
%
%

\section{Toric objects over the complex $K(J)$}\label{sec:toricandwedge}
    In this section, we study the relation of simplicial wedging and toric objects and prove Theorem~\ref{thm:mainthm}. To do so, we need the notion of ``projected characteristic map'' first of all.

    \begin{definition}
        Let $(K,\lambda)$ be a characteristic map of dimension $n$ and $\sigma\in K$ a face of $K$ {such that the set $\{\lambda(i)\mid i\in \sigma\}$ is unimodular}. Let $v$ be a vertex of $\link_K\sigma$. Then one maps $v$ to $[\lambda(v)]$ which is an element of the quotient lattice of $\Z^n$ by the sublattice generated by $\lambda(i)$, $i\in\sigma$. This map, denoted by $\proj_\sigma\lambda$, is called the \emph{projected characteristic map}, or shortly the \emph{projection}, of $\lambda$ with respect to $\sigma$.
    \end{definition}
    {
    There is a similar notion called the \emph{projected fans} (see Section 2 of \cite{HM03}). Note that projected characteristic maps generalize projected fans whenever it is applicable. We denote by $\proj_\sigma\Sigma$ the projected fan of a fan $\Sigma$ with respect to a face $\sigma$ of $K(\Sigma)$.
    }

    \begin{lemma}\label{lem:projection}
        Let $K$ be a fan-like sphere. Then for any proper face $\sigma$ of $K$, $\link_K\sigma$ is a fan-like sphere. If $(K,\lambda)$ is a complete non-singular characteristic map, then for any $\sigma$, its projection $(\link_K\sigma,\proj_\sigma\lambda)$ is also complete and non-singular. If $\lambda$ is fan-giving, so is $\proj_\sigma\lambda$.
    \end{lemma}
    \begin{proof}
        This is a topological toric version of projected fans and the proof is essentially the same. Since $K$ is fan-like, there exists a complete real (or rational) fan $\Sigma$ over $K$. Its projected fan is complete and therefore $\link_K\sigma$ is a fan-like sphere. Other assertions are obvious.
    \end{proof}
    We note that one can define projected topological fans in the same way. When $\sigma$ is a vertex, the projection $\proj_\sigma\lambda$ corresponds to a characteristic submanifold of $M(\lambda)$. We also remark that the above lemma shows that any multi-fan given by a complete characteristic map (or a complete topological fan) is complete.

    If $(K,\lambda)$ is an oriented complete characteristic map, then a projected characteristic map $(\link_K\sigma,\proj_\sigma\lambda)$ inherits an orientation so that
    \[
        w(I\cup\sigma) = w'(I),
    \]
    where $w$ is the weight function for $\lambda$ and $w'$ is that for $\proj_\sigma\lambda$. In this orientation convention, one immediately sees that an oriented complete characteristic map is positive if and only if every projection of it is positive.

    According to Section~7 of \cite{IFM12}, the orbit space $P$ of a topological toric manifold $M(K,\lambda)$ is a manifold with corners determined by $K$ whose face poset coincides with the inverse poset of $K$. We say two topological toric manifolds $\pi_i\colon M_i \to P$, $i=1,2$, are \emph{Davis-Januszkiewicz equivalent} or \emph{D-J equivalent} if there is an automorphism $\theta$ of $T^n$ and a homeomorphism $f\colon M_1\to M_2$ such that $f(g\cdot x) = \theta(g)\cdot f(x)$ and $\pi_2\circ f = \pi_1$. In other words, the following diagram
    \[
        \xymatrix{
                M_1 \ar[rr]^f \ar[dr] & & M_2 \ar[dl] \\
                 & P  }
    \]
    commutes.

    For two characteristic maps $\lambda_i=\lambda(M_i)$, $i=1,2$, we also say $(K,\lambda_1)$ and $(K,\lambda_2)$ are Davis-Januszkiewicz equivalent. Refer \cite{DJ91} and \cite{IFM12} for more details.

    Let $K$ be a fan-like sphere with $V(K)=[m]=\{1,\dotsc,m\}$. A characteristic map $\lambda\colon V(K) \to \Z^n$ can be regarded as an $n \times m$-matrix, called the \emph{characteristic matrix}, which is again denoted by $\lambda$. Each column is labeled by a vertex and the $i$-th column vector of the matrix $\lambda$ corresponds to $\lambda(i)$. Two characteristic maps $(K,\lambda_1)$ and $(K,\lambda_2)$ are Davis-Januszkiewicz equivalent if and only if there is a unimodular map sending the $i$-th column vector of $\lambda_1$ to that of $\lambda_2$ for all $i=1,\dotsc,m$. Therefore, elementary row operations on $\lambda$ preserve its D-J equivalence type and vice versa. We consider two characteristic maps are the same if they are D-J equivalent.

    \begin{example}
        Let $\wed_{1}(K)$ be the simplicial complex shown in Figure~\ref{fig:wedge} and $\lambda$ is defined by the characteristic matrix
        \[
            \lambda=\left(\,\begin{matrix}
                0 & 1 & 0 & -1 & -1 & 0 \\
                0 & 0 & 1 &  1 & 0  & -1 \\
                1 &-1 & 0 &  0& 0 & 0
            \end{matrix}\,\right)
        \]
        whose columns are labeled by the vertices ${1_1},{1_2},2,3,4,5$ respectively. That is, we define
        \begin{align*}
            \lambda({1_1}) &= (0,0,1) \\
            \lambda({1_2}) &= (1,0,-1) \\
            \lambda(2) &= (0,1,0) \\
            \lambda(3) &= (-1,1,0) \\
            \lambda(4) &= (-1,0,0) \\
            \lambda(5) &= (0,-1,0).
        \end{align*}
        Since $\lambda({1_1})$ is a coordinate vector, the projection $\proj_{{1_1}}\lambda$ is easily obtained by
        \[
            \proj_{{1_1}}\lambda = \left(\,\begin{matrix}
                 {1_2} & 2& 3 & 4 & 5 \\ \hline
                 1 & 0 & -1 & -1 & 0 \\
                 0 & 1 &  1 & 0  & -1
            \end{matrix}\,\right)
        \]
        where the first row is for indicating column labeling. To compute $\proj_{{3}}\lambda$, one should perform a row operation so that $\lambda(3)$ becomes a coordinate vector. Add the second row of $\lambda$ to the first one and one obtains
        \[
            \left(\,\begin{matrix}
                0 & 1 & 1 & 0 & -1 & -1 \\
                0 & 0 & 1 &  1 & 0  & -1 \\
                1 &-1 & 0 &  0& 0 & 0
            \end{matrix}\,\right).
        \]
        Since $\link_{K}\{3\}$ has vertices ${1_1},{1_2},2,4$, its characteristic matrix looks like
        \[
            \proj_{{3}}\lambda = \left(\,\begin{matrix}
                {1_1} & {1_2} & 2 & 4 \\ \hline
                0 & 1 & 1 & -1 \\
                1 & -1 & 0 & 0
            \end{matrix}\,\right).
        \]
    \end{example}

    \begin{proposition}\label{prop:chrwedge}
        Let $K$ be a fan-like sphere with vertex set $V$. For $v\in V$, let $\wed_v(K)$ be the wedge of {$K$} whose vertex set is $V\cup\{v_1,v_2\}\setminus\{v\}$ and let $\lambda$ be a complete characteristic map on $\wed_v(K)$ {such that $\{\lambda(v_1),\lambda(v_2)\}$ is a unimodular set}. Then $\lambda$ is non-singular if and only if $\proj_{v_1}\lambda$ and $\proj_{v_2}\lambda$ are non-singular. Furthermore, $\lambda$ is uniquely determined by $\proj_{v_1}\lambda$ and $\proj_{v_2}\lambda$. If the inherited orientations of $\proj_{v_1}\lambda$ and $\proj_{v_2}\lambda$ are positive, $\lambda$ is positively oriented.
    \end{proposition}

    \begin{proof}
        The non-singularity is easily verified since every maximal face of $\wed_v(K)$ contains $v_1$ or $v_2$. The positiveness also directly follows. Let us prove uniqueness of $\lambda$. We can assume that $v=1$. {By assumption,} the set $\{\lambda(1_1), \lambda(1_2)\}$ is unimodular, and for a suitable basis of $\Z^n$, we can assume that $\lambda(1_1)=e_1$ and $\lambda(1_2)=e_2$ where $e_i$ denotes the $i$-th coordinate vector of $\Z^n$. So the matrix for $\lambda$ has the form
        \begin{equation}\label{eqn:matrixwedge}
        \lambda=\left(
        \begin{array}{cc|ccc}
        1     &0      & a_{12} & \cdots & a_{1,m } \\
        0     &1      & a_{22} & \cdots & a_{2,m } \\ \hline
        0     &0      &     &        &     \\
        \vdots&\vdots &     & A&     \\
        0     &0      &     &        &
        \end{array}
        \right)_{(n+1)\times (m+1)}
        \end{equation}
        whose columns are labeled by the vertices $1_1, 1_2, 2,\dotsc,m$. Then the projected characteristic map $\proj_{1_{1}}\lambda$ is given by the matrix $N_2$ and $\proj_{1_2}\lambda$ is given by $N_1$, where

        \[
            N_j=\left(
            \begin{array}{c|ccc}
            1     & a_{j2} & \cdots & a_{j,m } \\ \hline
            0     &     &        &     \\
            \vdots&     & A&     \\
            0     &     &        &
            \end{array}
            \right),
        \]
        for $j=1,2$. It is obvious that once the matrices $N_1$ and $N_2$ are fixed, then $\lambda$ is forced to be unique.

    \end{proof}

    Notice that $\link_{\wed_1(K)}1_1$ and $\link_{\wed_1(K)}1_2$ can be naturally identified with $K$. Therefore, Proposition~\ref{prop:chrwedge} implies for any topological toric manifold $M=M(\lambda)$ over $K$ with
    \[
    \lambda=\left(
        \begin{array}{c|ccc}
        1     & a_{2} & \cdots & a_{m } \\ \hline
        0     &     &        &     \\
        \vdots&     & A&     \\
        0     &     &        &
        \end{array}
    \right),
    \]
    the matrix
    \begin{equation}\label{eqn:canonicalmatrix}
    \left(
        \begin{array}{cc|ccc}
        1     &0      & a_{2} & \cdots & a_{m } \\
        0     &1      & a_{2} & \cdots & a_{m } \\ \hline
        0     &0      &     &        &     \\
        \vdots&\vdots &     & A&     \\
        0     &0      &     &        &
        \end{array}
    \right)
    \end{equation}
    defines a topological {toric} manifold over $\wed_{1}(K)$. We write the new characteristic map by $\wed_1(\lambda)$ and the corresponding topological toric manifold by $\wed_1(M)$. This is called the \emph{canonical extension} of $M$ in \cite{Ewa86} when $\wed_1(M)$ is a toric manifold. The notation $\lambda(J)$ and $M(J)$ is used in \cite{BBCG10} since their underlying complex is $K(J)$ in Section~\ref{sec:wedge}. Following \cite{Ewa86}, we call $\wed_1(\lambda)$ or $\wed_1(M)$ the \emph{canonical extension} or the \emph{trivial wedging}. Let us briefly introduce the notation $M(J)$ of \cite{BBCG10}. Let $M(\lambda)$ be a topological toric manifold over $K$ with vertex set $V(K)=[m]=\{1,\ldots,m\}$ and $J=(a_1,\ldots,a_m)\in\N^m$. Although $M(J)$ is defined for arbitrary $J$, we will do it only when $J=(1,\ldots,1,2,1\ldots,1)$, where the $k$-th entry of $J$ is 2. Then $K(J) = \wed_k (K)$ and the matrix
    \begin{equation}\label{eqn:mj}
    \lambda'=\left(\begin{array}{c|ccccccc}
        1 & 0 & \cdots & 0 & -1 & 0 & \cdots & 0 \\ \hline
        0 &  &  & & & & & \\
        \vdots &\lambda_1 & \cdots & & \lambda_k & & \cdots & \lambda_m \\
        0 & & & & & & & \\
    \end{array}\right),
    \end{equation}
    where $\lambda_i$ is the $i$-th column of $\lambda$, is a characteristic matrix for $\wed_{k}(K)$ with respect to the ordering of facets $k_1, 1,2,\ldots,k-1,k_2,{k+1},\ldots,m$, defining $M(J)$ over $K(J)$. Without loss of generality, we can assume $k=1$ and $\lambda_k=\lambda_1$ is the first coordinate vector $e_1$. Then the matrix above is re-written like the following:
    \begin{equation*}
    \left(
        \begin{array}{cc|ccc}
        1     &-1      & 0 & \cdots & 0 \\
        0     &1      & a_{2} & \cdots & a_{m } \\ \hline
        0     &0      &     &        &     \\
        \vdots&\vdots &     & A&     \\
        0     &0      &     &        &
        \end{array}
    \right).
    \end{equation*}
    Now adding the second row to the first one obtains the wanted result. This shows that $M(J)$ of \cite{BBCG10} can be obtained from consecutive trivial wedgings of a topological toric manifold $M$. This construction can be iterated for general $K(J)$, defining the characteristic map $\lambda(J)$ for $M(J)$.

    Proposition~\ref{prop:chrwedge} certainly extends for general $K(J)$ in place of $\wed_v(K)$. Let $V(K)=[m]$ and $J=(a_1,\dotsc,a_m)$ as before. Repeatedly applying Proposition~\ref{prop:chrwedge}, we reach a subcomplex of $K(J)$ naturally isomorphic to $K$. Let us describe the subcomplex. Let $\pi\colon V(K(J)) \to V(K)$ be the natural surjective map given by $i_k\mapsto i$ (actually this map induces a simplicial map between two simplicial complexes). Let $s$ be a section map $s\colon V(K)\to V(K(J))$ such that $\pi\circ s = id_{V(K)}$. It is obvious that the image of $s$ induces a subcomplex denoted by $K_s$, naturally isomorphic to $K$, reminding the definition of $K(J)$ in Section~\ref{sec:wedge}. Moreover, the subcomplex is the link of the simplex $V(K(J))\setminus \operatorname{im} s$ (observe this set does not contain any non-face as a subset). By Lemma~\ref{lem:projection}, we have the following:
    \begin{corollary}\label{cor:projectionofKJ}
        Let the setting be above and let $\Lambda$ be a {non-singular} characteristic map over $K(J)$. Then the projection of $\Lambda$ with respect to the simplex $V(K(J))\setminus \operatorname{im}s$ is non-singular for every section $s\colon V(K)\to V(K(J))$. Furthermore, $\Lambda$ is uniquely determined by projections of $\Lambda$ on $K_s$. If every such projection is D-J equivalent to a characteristic map $\lambda$ over $K \cong K_s$, then $\Lambda = \lambda(J)$.
    \end{corollary}
    In particular, we obtain the following:
    \begin{corollary}
        The simplicial complex $K(J)$ admits a topological toric manifold if and only if $K$ does.
    \end{corollary}
    Corollary~\ref{cor:projectionofKJ} shows that ones can find all topological toric manifolds over $K(J)$ provided they know every topological toric manifold over $K$. As we have seen, if $K$ is polytopal then so is $K(J)$ and this means that an analogue of the above corollary also holds for quasitoric manifolds. Every argument so far applies to the $\Z_2$-version of characteristic maps and thus real topological toric manifolds and small covers.

    Our next step is for fan-giving characteristic maps and their associated toric manifolds. If $(K,\lambda)$ is fan-giving, we can assume that $|K|$ is oriented such that $\lambda$ is a positive characteristic map.

    \begin{proposition}\label{prop:FanandWedge}
        Under the setting of Proposition~\ref{prop:chrwedge}, a complete characteristic map $\lambda$ is fan-giving if and only if $\proj_{v_1}\lambda$ and $\proj_{v_2}\lambda$ are fan-giving.
    \end{proposition}
    \begin{proof}
        We only need to show the `if' part. Suppose the dimension of $\lambda$ is $n+1$. Then we have a map $f_\lambda\colon |\wed_v(K)|\to S^n$ as we defined in Subsection~\ref{subsec:charandmultifans}. By Lemma~4.2 of \cite{IM12}, the restriction $f_\lambda|_{\operatorname{St}\{v_i\}}$ $(i=1,2)$ is an embedding of an $n$-disc into $S^n$, where $\operatorname{St}\{v_i\}$ denotes the open star of $v_i$ in $\wed_v(K)$. Observe that $\operatorname{St}\{v_1\} \cap \operatorname{St}\{v_2\} $ contains a maximal simplex $\{v_1,\,v_2\}\cup\tau$, where $\tau$ is a maximal simplex of $\link_K\{v\}$. By combining these clues and the fact $\overline{\operatorname{St}\{v_1\}} \cup \overline{\operatorname{St}\{v_2\}}=\wed_v(K)$, one concludes that if $x\in\operatorname{int}|\{v_1,\,v_2\}\cup\tau|$, $y\in|\wed_v(K)|$, and $f_\lambda(x)=f_\lambda(y)$, then $y=x$. Apply Proposition~\ref{prop:toddandfangiving} to finish the proof.
    \end{proof}

    Combining Proposition~\ref{prop:chrwedge} and Proposition~\ref{prop:FanandWedge}, we obtain the main result which we restate below:
    \begin{theorem}
        Let $K$ be a fan-like simplicial sphere and $v$ a given vertex of $K$. Let  $(\wed_v(K), \lambda)$ be a characteristic map and let $v_1$ and $v_2$ be the two new vertices of $\wed_v(K)$ created from the wedging. {Let us assume that $\{\lambda(v_1),\lambda(v_2)\}$ is a unimodular set.} Then $\lambda$ is uniquely determined by the projections $\proj_{v_1}\lambda$ and $\proj_{v_2}\lambda$. Furthermore,
        \begin{enumerate}
            \item $\lambda$ is non-singular if and only if so are $\proj_{v_1}\lambda$ and $\proj_{v_2}\lambda$.
            \item $\lambda$ is positively oriented if and only if so are $\proj_{v_1}\lambda$ and $\proj_{v_2}\lambda$.
            \item $\lambda$ is fan-giving if and only if so are $\proj_{v_1}\lambda$ and $\proj_{v_2}\lambda$.
        \end{enumerate}
    \end{theorem}
    \begin{remark}
        An omnioriented quasitoric manifold admits an equivariant almost complex structure if and only if its (oriented) characteristic map is positively oriented by Kustarev \cite{Kus09}. Combining this with the above theorem we can classify equivariantly almost complex quasitoric manifolds over $\wed_F(P)$ provided we know all such manifolds over $P$. The authors think that an analogous result is also true for equivariantly almost complex topological toric manifolds.
        Note also that every toric manifold is a complex manifold with an equivariant complex structure and therefore admits an equivariant almost complex structure.
    \end{remark}

\section{Wedge operations and projectivity of toric varieties}\label{sec:wedgeandprojectivity}

    Let $X:=(x_1,\dotsc,x_m)\in (\R^n)^m$ be a finite sequence of vectors in $\R^n$ which linearly spans $\R^n$. Consider the space of linear dependencies of $X$, which is an $(m-n)$-dimensional space
    \[
        \{(\alpha_1,\dotsc,\alpha_m)\in \R^m \mid \alpha_1x_1+\dotsb+\alpha_mx_m = 0 \}.
    \]
    Choose one of its basis and write down them as rows of a matrix
    \[
        \begin{pmatrix}
            \alpha_{11} & & \cdots & & \alpha_{1m} \\
            \vdots & & & & \vdots \\
            \alpha_{m-n,1} & & \cdots & & \alpha_{m-n,m}
        \end{pmatrix}_{(m-n)\times m}=:(\overline{x}_1,\dotsc,\overline{x}_m).
    \]
    The sequence $\overline{X}$ of column vectors of this matrix is called a \emph{linear transform} of $X$. Note that Shephard \cite{She71} used the term \emph{linear representation}. A linear transform is not uniquely determined and it is defined up to basis change. When each vector is regarded as a column vector, the matrices $X=(x_1,\dotsc,x_m)$ and $\overline{X}=(\overline{x}_1,\dotsc,\overline{x}_m)$ have the relation $X\overline{X}^T = O$. Its transpose is again $\overline{X}X^T = O$. In other words, their row spaces are orthogonal to each other and span the entire space $\R^m$. Therefore if $\overline{X}$ is a linear transform of $X$, then $X$ is a linear transform of $\overline{X}$.

    For a subsequence $Y=(x_{i_1},\dotsc,x_{i_k})$ of $X$, one writes ${\overline{X}|_Y} :=(\overline{x}_{i_1},\dotsc,\overline{x}_{i_k})$.  If $X\setminus Y$ positively spans a face of the polyhedral cone $\pos X$, we say that $Y$ is a \emph{coface} of $X$. The following result is fundamental in the theory of Gale transform.
    \begin{theorem}\cite[Theorem~II.4.14]{Ewa96}\label{thm:cofacelinear}
        Let $\overline{X}$ be a linear transform of $X$ and $Y$ a subsequence of $X$. Then $Y$ is a coface of $X$ if and only if
        \[
            0 \in \relint \conv {\overline{X}|_Y}
        \]
        ($\relint$ = relative interior of, $\conv$ = convex hull of).
    \end{theorem}
    An immediate corollary follows:
    \begin{corollary}\label{cor:coneandpospan}
        The set $\pos X$ is strongly convex if and only if $\overline{X}$ positively spans $\R^{m-n}$.
    \end{corollary}
    \begin{proof}
        Exchange the roles of $X$ and $\overline{X}$ and then a proof is straightforward once one notices that $\pos\overline{X}$ has no proper face.
    \end{proof}
    We need also the next lemma.
    \begin{lemma}\label{lem:hyperplaneandzerosum}
        A linear transform $\overline{X}$ of $X$ satisfies $\overline{x}_1+\dotsb+\overline{x}_m = 0 $ if and only if the points $x_i$ lie in a hyperplane $H$ of $\R^n$ for which $0 \notin H$.
    \end{lemma}
    \begin{proof}
        Let $a_1,\dotsc,a_m\in\R$ be a linear relation such that
        \[
            a_1x_1+\dotsb+a_mx_m = 0
        \]
        and
        \[
            a_1+\dotsb+a_m = 0.
        \]
        Such a linear relation is called an \emph{affine relation}. Put $y_i := x_i - x_m$, $i=1,\dotsc,m-1$. Then $a_m = -a_1-\dotsb-a_{m-1}$ and
        \[
            a_1y_1+\dotsb+a_{m-1}y_{m-1} = 0.
        \]
        There are $m-n$ such relations and thus $\operatorname{span}\{y_i\}$ has dimension $n-1 = m-1-(m-n)$. We obtain the hyperplane $H$ by translating $\operatorname{span}\{y_i\}$ by $x_m$. The converse is proven by reversing all the argument.
    \end{proof}
    Note that one can assume that $H$ is the hyperplane of points whose last coordinate is 1 since we can take $(1,\dotsc,1)$ for a linear dependency of $\overline{X}$. In general, for any strongly convex cone $C$, there is a hyperplane $H$ which does not intersect the origin and $C\cap H = P$ is a convex polytope which has the same face poset with $C$. Now we are ready to define the Gale transform.
    \begin{definition}
        Let $X=(x_1,\dotsc,x_m)\in(\R^n)^m$ be a sequence of points affinely spanning $\R^n$ and $f\colon\R^n \hookrightarrow \R^{n+1}$ be an embedding defined by $f(v)=(v,1)$. Then a \emph{Gale transform} of $X$, denoted by $X'=(x'_1,\dotsc,x'_m)$, is a linear transform of $f(X)$ in $\R^{n+1}$, which is defined in $\R^{m-n-1}$.
    \end{definition}
    The set $P=\conv X\subset \R^n$ is surely an $n$-polytope. In some sense $X$ can be regarded as the ``vertex set'' of $P$, even though $X$ can have a multiple point or a point on the relative interior of a face of $P$. In the latter case, the point does not represent a vertex of $P$ and we call it a \emph{ghost vertex} of $P$. If $X$ is indeed the vertex set of $P$, then we call $X'$ a Gale transform of $P$. Corollary~\ref{cor:coneandpospan} and Lemma~\ref{lem:hyperplaneandzerosum} implies that $X'$ positively spans $\R^{m-n-1}$ and $x'_1+\dotsb+x'_m = 0 $. The term coface is used analogously. We sometimes say that $Y$ is a coface of $P=\conv X$ if $\conv(X\setminus Y)$ is a face of $P$. An analogue of Theorem~\ref{thm:cofacelinear} for convex polytopes is as follows, which is very useful to study polytopes when $m-n$ is small (at most 4).
    \begin{corollary}\label{cor:galeforpolytopes}
        Let $P = \conv X$ be an $n$-polytope with $m$ vertices. Then the subsequence $Y$ of $X$ is a coface of $P$ if and only if
        \[
            0\in \relint \conv {X'|_Y}.
        \]
        The polytope $P$ is simplicial if and only if $\conv X'|_Y$ is a simplex of dimension $m-n-1$ for all minimal coface $Y$ of $P$.
    \end{corollary}
    \begin{proof}
        The first assertion has been already shown. Suppose that $P$ is a simplicial $n$-polytope with $m$ vertices. Then each maximal face of $P$ has $n$ vertices and therefore for each minimal coface $Y$, ${X'|_Y}$ has cardinality $m-n$ in $\R^{m-n-1}$. Suppose that $\conv {X'|_Y}$ is not a simplex. Then its dimension is less than $m-n-1$. Recall that the famous Carath\'{e}odory's theorem states that if a point $v \in \R^d$ lies in $\conv S$ for a set $S$, then there is a subset $T$ of $S$ consisting of $d+1$ or fewer points such that $v \in \conv T$. Therefore one can apply Carath\'{e}odory's theorem to see that $Y$ is not minimal, which is a contradiction. The converse is immediate since every coface of $P$ has cardinality $m-n$.
    \end{proof}
        In general, for any sequence $X'=(x'_1,\dotsc,x'_m)$ of points of $\R^{m-n-1}$, one can define the poset which consists of subsequences of the form $X\setminus Y$, where $X=(x_1,\dotsc,x_m)$ is a sequence of symbols and $Y$ is a subsequence of $X$ such that $0\in\relint \conv {X'|_Y}$. If this poset coincides with a face poset of a polytope $P$, then we call $X'$ a \emph{Gale diagram} of $P$. By definition, a Gale transform is a Gale diagram. Two Gale diagrams are called \emph{isomorphic} if they share the same face poset. We will see an application of Gale transform in next section. We would need a lemma there, so we introduce it here. For a polytope $P$ and its given vertex $x$, recall that a \emph{vertex figure} of $P$ at $x$ is a polytope defined by $P\cap H$ where $H$ is a hyperplane separating $x$ from the other vertices of $P$. A vertex figure is uniquely determined up to combinatorial equivalence of polytopes, so let us call its equivalence class ``the'' vertex figure at $x$.
    \begin{lemma}\label{lem:GaleandVertexfigure}
        Let $X=(x_1,\dotsc,x_m)$ be a sequence of points and let $X' = (x'_1,\dotsc,x'_m)$ be a Gale diagram of the simplicial polytope $P^* = \conv X$. Then the subsequence
        \[
            X'' = (x'_1,x'_2,\dotsc,x'_{i-1},x'_{i+1},\dotsc,x'_m)
        \]
        is a Gale diagram for the vertex figure of $P^*$ at $x_i$.
    \end{lemma}
    \begin{proof}
        Observe that the vertex figure at $x_i$ has boundary complex $\link_{K}\{x_i\}$. It is immediate to see that the face complex of $X''$ and the complex $\link_{K}\{x_i\}$ are the same.
    \end{proof}
    In the language of simple polytopes, the vertex figure at $x_i$ corresponds to the facet of $(P^*)^* = P$ dual to $x_i$. Note that the lemma above does not hold for general face figures because it is different from the link. Throughout this paper, we consider only Gale diagrams for simplicial polytopes. When we say about a Gale diagram for a simple polytope $P$, then actually it means a Gale diagram for the dual polytope $P^*$ which is simplicial.

    Returning for projectivity of fans, we consider an ``inverse'' of the Gale transform in some sense. By Corollary~\ref{cor:coneandpospan}, any linear transform of a positively spanning sequence is a strongly convex cone. Let $X=(x_1,\dotsc,x_m)$ be a sequence positively spanning $\R^n$. Then by Corollary~\ref{cor:coneandpospan}, the set $\pos\overline{X}$ is a strongly convex cone $C$. Let $H$ be any hyperplane such that $H\cap C$ is an $(m-n-1)$-polytope $\widehat{P}$. For each $\overline{x}_i \in \overline{X}$ consider the ray $r(\overline{x}_i) = \{a\overline{x}_i\mid a>0\}$ and let this ray meet $H$ in $\widehat{x}_i$. Then the sequence $\widehat{X}:=(\widehat{x}_1,\dotsc,\widehat{x}_m)$ in $H$ is called a \emph{Shephard diagram} of $X$. Observe that the inverse operation of the Gale transform gives a Shephard diagram of $X$. The Shephard diagram $\widehat{X}$ is independent of the choice of $x_i$'s. To see this, let $a_i$ be nonzero reals for $i=1,\dotsc,m$ and observe that a linear transform of $(a_1x_1,\dotsc,a_mx_m)$ is $(a_1^{-1}\overline{x}_1,\dotsc,a_m^{-1}\overline{x}_m)$. The next theorem indicates the relation between Gale diagrams and Shephard diagrams.
    \begin{theorem}\cite{She71}\label{thm:shepharddiagram}
        Let $\mu:=(\mu_1,\dotsc,\mu_m)$ be a vector of positive real numbers and put $P_\mu:=\conv (\mu_1x_1,\dotsc,\mu_mx_m)$. Then $\widehat{X}$ is a Gale diagram of $P_\mu$ if we select some suitable point $z\in H$ as origin.
        Conversely, if we select any $z\in \operatorname{int}\conv \widehat{X}$ for the origin, then $\widehat{X}$ is a Gale diagram of $P_\mu$ for some $\mu$.
    \end{theorem}

    Let $\Sigma$ be a complete fan of dimension $n$. We choose a point $x_i$, $i=1,\dotsc,m$, from each 1-cone of $\Sigma$. Then, by completeness, the sequence $X=(x_1,\dotsc,x_m)$ positively spans $\R^n$. The sequence $\widehat{X}$ is sometimes called a \emph{Shephard diagram for the fan $\Sigma$}. Let $Y$ be a subsequence of $X$ and suppose that $\pos (X\setminus Y)$ is a face of $\Sigma$. Then $Y$ is called a \emph{coface} of $\Sigma$. The next theorem is Shephard's criterion for projectivity of fans which is easily induced from Theorem~\ref{thm:shepharddiagram}.

    \begin{theorem}[Shephard's criterion]\cite{She71,Ewa86}
        A complete fan $\Sigma$ is strongly polytopal if and only if
        \[
            S(\Sigma,\widehat{X}) := \bigcap_{Y\colon \text{coface of }\Sigma}\relint \conv {\widehat{X}|_Y} \ne\varnothing.
        \]
    \end{theorem}

    Let $K$ be a fan-like sphere with $V(K)=[m]$ and $\Sigma$ be a complete fan over $\wed_1(K)$. Assume that $V(\wed_1(K))=[m]\cup\{0\}$ by renaming $1_1$ to $0$ and $1_2$ to $1$. Choose a point $x_i$ from each 1-cone corresponding to $i\in V(\wed_1(K))$.
    \begin{proposition}\label{prop:ShephardandWedge}
        In above setting, let $\widehat{X}=(\widehat{x}_0,\widehat{x}_1,\dotsc,\widehat{x}_m)$ be a Shephard diagram for $\Sigma$. Then the subsequences $\widehat{X} \setminus(\widehat{x}_0) = (\widehat{x}_1,\widehat{x}_2,\dotsc,\widehat{x}_m)$ and $\widehat{X} \setminus(\widehat{x}_1) = (\widehat{x}_0,\widehat{x}_2,\dotsc,\widehat{x}_m)$ are Shephard diagrams for the projected fans $\proj_0\Sigma$ and $\proj_1\Sigma$ respectively. Furthermore,
        \[
            S(\Sigma,\widehat{X}) = S(\proj_0\Sigma, \widehat{X}\setminus(\widehat{x}_0)) \cap S(\proj_1\Sigma, \widehat{X}\setminus(\widehat{x}_1)) .
        \]
    \end{proposition}
    \begin{proof}
        To compute $\widehat{X}$, let $x_i$ be the $i$-th column vector of the following matrix
        \begin{equation*}
            X=\left(
                \begin{array}{cc|ccc}
                q_0     &0      &   & a &   \\
                0     &q_1      &   & b &   \\ \hline
                0     &0      &     &        &     \\
                \vdots&\vdots &     & A&     \\
                0     &0      &     &        &
                \end{array}
            \right)_{(n+1)\times(m+1)},
        \end{equation*}
        where $q_0,q_1>0$ are positive reals, $a$ and $b$ are row vectors of dimension $m-1$, and each column is labeled by the vertices $0,1,\dotsc,m$ (See \eqref{eqn:matrixwedge}). Moreover, one can assume that $x_0+x_1+\dotsb+x_m = 0$ (that is why $q_0$ and $q_1$ are not 1) and thus one can compute a Shephard diagram $\widehat{X}$ by the inverse operation of the Gale transform. First, let $\widehat{A}$ be a matrix indicating a Shephard diagram of columns of $A$. More precisely, we choose $\widehat{A}$ such that the following identity
        \[
            A\cdot\left(\,\begin{matrix}
                 & \widehat{A} &  \\ \hline
                 1 & \cdots & 1
            \end{matrix}\,\right)^T = O,
        \]
        which is possible since the sum of columns of $A$ is zero. Now observe that a Shephard diagram of $X$ is written as
        \[
            \widehat{X}=\left(\begin{array}{cc|ccc}
            \alpha_1 & \beta_1 &  & & \\
            \vdots & \vdots & & \widehat{A} & \\
            \alpha_{m-n-1} & \beta_{m-n-1} & & &
            \end{array}\right)_{(m-n-1)\times(m+1)},
        \]
        where $\alpha_i$ and $\beta_i$ are real numbers satisfying
        \[
            q_0\alpha_i + a\cdot \widehat{A}^i = 0,\qquad i=1,\dotsc,m-n-1,
        \]
        and
        \[
            q_1\beta_i + b\cdot \widehat{A}^i = 0,\qquad i=1,\dotsc,m-n-1
        \]
        ($\widehat{A}^i$ denotes the $i$-th row of $\widehat{A}$).

        On the other hand, it is easy to see the matrix for $\proj_1\Sigma$ is
        \[
            \left(
                \begin{array}{c|ccc}
                q_0           &   & a &   \\ \hline
                0           &     &        &     \\
                \vdots &     & A &     \\
                0          &     &        &
                \end{array}
            \right)_{n\times m},
        \]
        the sum of columns of which is still zero and therefore its Shephard diagram is
        \[
            \left(\begin{array}{c|ccc}
            \alpha_1 &  & & \\
            \vdots & & \widehat{A} & \\
            \alpha_{m-n-1} & & &
            \end{array}\right)_{(m-n-1)\times m}.
        \]
        The same goes for $\proj_0\Sigma$ and the proof is done. The last identity is obvious by observing cofaces of $\Sigma$.
    \end{proof}
    When $M$ is a compact toric orbifold, the proposition above provides an alternative proof of the fact \cite{Ewa86} that the canonical extension $M(J)$ is projective if and only if $M$ is projective since $S(\proj_0\Sigma, \widehat{X}\setminus(\widehat{x}_0)) = S(\proj_1\Sigma, \widehat{X}\setminus(\widehat{x}_1)) = S(\Sigma, \widehat{X})$.

    In the proof of Proposition~\ref{prop:ShephardandWedge}, we have essentially shown the following:
    \begin{proposition}\label{prop:Shephardandproj}
        Let $K$ be a fan-like sphere with $V(K) = [m]$ and $\Sigma$ be a complete fan over $K$ with $m$ 1-cones. Choose a point $x_i$ from each 1-cone corresponding to $i\in [m]$ and let $\widehat{X} = (\widehat{x}_1,\dotsc,\widehat{x}_m)$ be a Shephard diagram for $\Sigma$. Then the subsequence $\widehat{X}\setminus(\widehat{x}_i)$ is a Shephard diagram for $\proj_{i}\Sigma$ for any $i\in[m]$.
    \end{proposition}
    Be cautious that the underlying complex of $\proj_{i}\Sigma$ is $\link_K{\{i\}}$ and in general its vertices do not bijectively correspond to entries of $\widehat{X}\setminus(\widehat{x}_i)$, causing ghost vertices. But we still have no problem to use $\widehat{X}\setminus(\widehat{x}_i)$ to apply Shephard's criterion to determine if $\proj_{i}\Sigma$ is strongly polytopal.
We continue to study projectivity of toric varieties in Section~\ref{sec:projectivity}.

%


\section{Application: Classification of toric varieties}\label{sec:applications}

    Let $M = M(K,\lambda)$ be a topological toric manifold of dimension $2n$. One can think about two possible applications of the main result:

    \begin{enumerate}
        \item $K = \star_{i=1}^k\partial\Delta^{n_i}$ is the join of boundaries of simplices. Note that the dual of $K$ is the boundary of the simple polytope $\prod_{i=1}^k\Delta^{n_i}$.
        \item $K$ is a simplicial sphere of dimension $n-1$ with at most $n+3$ vertices.
    \end{enumerate}

    Let $ P = \prod_{i=1}^k\Delta^{n_i}$ be a product of simplices. In fact, quasitoric manifolds and (especially toric manifolds) over $P$ is already studied by \cite{CMS10}.

    \begin{definition}
        The \emph{generalized Bott tower} is the following sequence of projective bundles
        \[
            B_\ell \stackrel{\pi_\ell}{\longrightarrow } B_{\ell-1}\stackrel{\pi_{\ell-1}}{\longrightarrow }\dotsb \stackrel{\pi_2}{\longrightarrow }B_1\stackrel{\pi_1}{\longrightarrow } B_0 =\{\text{a point}\},
        \]
        where $B_i$ for $i = 1,\dotsc,\ell$ is the projectivization of the Whitney sum of $n_i+1$ $F$-line bundles over $B_{i-1}$ where $F=\C$ or $\R$. Each $B_i$ is called a \emph{generalized Bott manifold} over $F$ of stage $i$. If all $n_i$ is equal to $1$, then we call the sequence a \emph{Bott tower} and each $B_i$ a \emph{Bott manifold}.
    \end{definition}
    Every generalized Bott manifold is actually a quasitoric manifold over $ P = \prod_{i=1}^k\Delta^{n_i}$ which is a smooth projective toric variety. Conversely, every toric manifold (in fact, every quasitoric manifold admitting an equivariant almost complex structure) over $P$ becomes a generalized Bott manifold.
    \begin{lemma}
        Let $K$ and $L$ be simplicial complexes whose vertex sets are $\{v_1,\dotsc,v_m\}$ and $\{w_1,\dotsc,w_\ell\}$ respectively. Let $J=(a_1,\dotsc,a_m)$ and $J'=(b_1,\dotsc,b_\ell)$ be vectors whose entries are positive integers. Then
        \[
            K(J)\star L(J') = (K\star L)(J\cup J'),
        \]
        where $J\cup J' = (a_1,\dotsc,a_m,b_1,\dotsc,b_\ell)$.
    \end{lemma}
    \begin{proof}
        A proof is straightforward once observing that a minimal non-face of
        \[
            K \star L = \{ \sigma \cup \tau \mid \sigma \in K, \tau \in L\}
        \]
        is a minimal non-face of $K$ or $L$.
    \end{proof}
    Hence the simplicial complex $\star_{i=1}^k\partial\Delta^{n_i}$ is obtained by wedge operations from the complex $\star_{i=1}^k\partial\Delta^1$ which is the boundary of the $k$-cross polytope. Therefore, in the language of simple polytopes, any product of simplices is obtained by polytopal wedge operations from the $k$-cube $I^k$. Let $M = M(P,\lambda)$ be a quasitoric manifold admitting an equivariant almost complex structure. Then its projection with respect to a $k$-cube $I^k$ must be a Bott manifold. Conversely, if every projection of $M = M(P,\lambda)$ with respect to a $k$-cube is a toric manifold, then it can be shown that $M$ is a generalized Bott manifold although the proof is omitted. For further study of (generalized) Bott manifolds, see \cite{GK94}, \cite{MP08}, and \cite{CMS10}.

    Next, let us consider the second case. It was proved by Mani \cite{Man72} that every simplicial $(n-1)$-sphere with at most $n+3$ vertices is polytopal. Hence in this case every topological toric manifold over $K$ is a quasitoric manifold and we are left with simplicial $n$-polytopes $P^*$ with at most $n+3$ vertices which is classified using the Gale diagram introduced in previous section. In this case, its Gale diagram lies in $\R^2$. In the language of simple polytopes, its dual is a simple $n$-polytope $P$ with $\le n+3$ facets. Let $P_{2k-1}$ be a regular $(2k-1)$-gon in $\R^2$ with center at the origin $O$ and vertex set $\mathcal{V}=\{v_1,\ldots,v_{2k-1}\}$, where $v_i$'s are labeled in counterclockwise order. For convenience we further assume that $v_i$'s are on the unit circle. For a given surjective map $\phi: [n+3] \to \mathcal{V}$, we have a sequence of points $X'=(x'_1,\dotsc,x'_{n+3})$ such that $x'_i = \phi(i)$. We call $X'$ a \emph{standard Gale diagram} in $\R^2$. Observe that the face poset defined by $X'$ is a simplicial complex and hence the corresponding polytope $P^\ast$ is simplicial. Let $K$ be the simplicial complex given by $X'$. Note that $K$ is a boundary complex for a simple polytope $P$. Recall that
    \[
        I\text{ is a face of }K \Longleftrightarrow O\in\conv\{\phi(i)\mid i\in[n+3]\setminus I\}.
    \]
    For $1\le i\le 2k-1$, let $a_i$ be the cardinality of $\phi^{-1}(v_i)$. The standard Gale diagram is determined by $a_i$'s up to symmetry of $P_{2k-1}$. A polygon $P_{2k-1}$ whose vertices $v_i$ are numbered by a positive integer $a_i$ is again called a standard Gale diagram. We sometimes abuse the term Gale diagram for standard Gale diagram. See Figure~\ref{fig:galediagrams} for an illustration.

    It is a classical result that every simple $n$-polytope with not more than $n+3$ facets has a corresponding standard Gale diagram on $\R^2$. Moreover, two simple $n$-polytopes with $n+3$ facets are combinatorially equivalent if and only if their standard Gale diagrams coincide after an orthogonal linear transform of $\R^2$ onto itself. Note that $2k-1=3$ and $a_i = 1$ for some $i$ if and only if $P$ has $n+1$ or $n+2$ facets.

    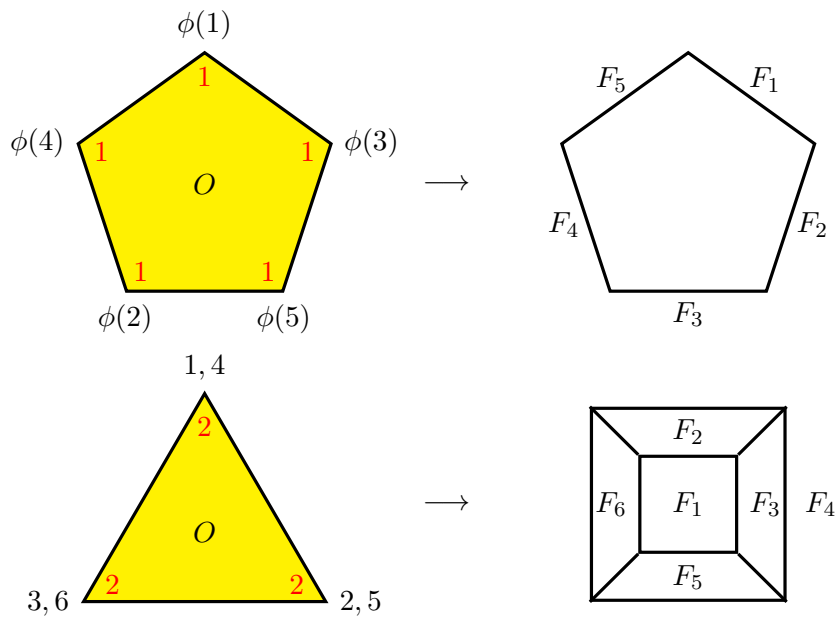
\begin{figure}
        \begin{tikzpicture}[scale=.53, very thick]
            \node[name=p, shape=regular polygon, regular polygon sides=5, inner sep=1cm, draw, fill=yellow] at (-6,8) {};
            \node[name=p1, shape=regular polygon, regular polygon sides=5, inner sep=.8cm] at (-6,8) {};
            \node[name=pp, shape=regular polygon, regular polygon sides=5, inner sep=1cm, draw] at (6,8) {};
            \node[name=pp1, shape=regular polygon, regular polygon sides=5, inner sep=1.2cm] at (6,8) {};
            \node[name=t, regular polygon, regular polygon sides=3, inner sep=.65cm, draw, fill=yellow] at (-6,-.7) {};
            \node[name=t1, regular polygon, regular polygon sides=3, inner sep=.48cm] at (-6,-.7) {};
            \node[name=s, regular polygon, regular polygon sides=4, inner sep=.9cm, draw] at (6,0) {};
            \node[name=ss, regular polygon, regular polygon sides=4, inner sep=.45cm, draw] at (6,0) {};
            \foreach \k in {corner 1, corner 2, corner 3, corner 4}
            \draw (s.\k) -- (ss.\k);
            \foreach \index in {corner 1, corner 2, corner 3, corner 4, corner 5}
            \draw [shift=(p1.\index)] node {\red{$1$}};
            \draw [shift=(p.center)] node {$O$};
            \draw [shift=(p.corner 1)] node[above] {$\phi(1)$};
            \draw [shift=(p.corner 5)] node[right] {$\phi(3)$};
            \draw [shift=(p.corner 4)] node[below] {$\phi(5)$};
            \draw [shift=(p.corner 3)] node[below] {$\phi(2)$};
            \draw [shift=(p.corner 2)] node[left ] {$\phi(4)$};
            \draw (0,8) node {$\longrightarrow$};
            \draw [shift=(pp1.side 5)] node {$F_1$};
            \draw [shift=(pp1.side 4)] node {$F_2$};
            \draw [shift=(pp1.side 3)] node {$F_3$};
            \draw [shift=(pp1.side 2)] node {$F_4$};
            \draw [shift=(pp1.side 1)] node {$F_5$};
            \foreach \j in {corner 1, corner 2, corner 3}
            \draw [shift=(t1.\j)] node {\red{$2$}};
            \draw [shift=(t.center)] node {$O$};
            \draw [shift=(t.corner 1)] node[above] {$1,4$};
            \draw [shift=(t.corner 3)] node[right] {$2,5$};
            \draw [shift=(t.corner 2)] node[left ] {$3,6$};
            \draw (0,0) node {$\longrightarrow$};
            \draw [shift=(s.center)] node {$F_1$};
            \draw [shift=(ss.side 1)] node[above] {$F_2$};
            \draw [shift=(ss.side 4)] node[right] {$F_3$};
            \draw [shift=(ss.side 3)] node[below] {$F_5$};
            \draw [shift=(ss.side 2)] node[left ] {$F_6$};
            \draw [shift=(s.side 4),xshift=5pt] node[right] {$F_4$};
        \end{tikzpicture}
        \caption{Standard Gale diagrams and the corresponding polytopes: $[1,1,1,1,1]$ and $[2,2,2]$.}\label{fig:galediagrams}
    \end{figure}

    Let us denote by $[a_1,\ldots,a_{2k-1}]$ the simple polytope whose Gale diagram is the $(2k-1)$-gon whose vertices are labeled by $(a_1,\ldots,a_{2k-1})$. For example, $[n_1+1,n_2+1,n_3+1]$ is the product of simplices $\prod_{i=1}^3\Delta^{n_i}$, $n_i\ge 0$.

    We could apply Proposition~\ref{prop:chrwedge} to classify topological toric manifolds over the polytope $[a_1,\dotsc,a_{2k-1}]$, but we restrict our interest to toric manifolds for now. First, by Lemma~\ref{lem:pjandgale}, every simple $n$-polytope with $n+3$ facets is obtained by consecutive polytopal wedgings from the cube $[2,2,2]$ or
    \[
        P_{[2k-1]} := [\underbrace{1,\dotsc,1}_{2k-1}],\quad k\ge 3.
    \]
    \begin{convention}
        For a given polytope $P$ and vector $J$, one needs a suitable labeling of facets of $P$ for $P(J)$ to be well defined. Note that every vertex labeling of the standard Gale diagram induces a facet labeling of $P_{[2k-1]}$. If there is no comment about it, the default convention for $P_{[2k-1]}$ is that we label the vertices of the standard Gale diagram in counterclockwise order.
    \end{convention}

    \begin{lemma}\label{lem:pjandgale}
        Let $P$ be the simple polytope $P_{[2k-1]}$ with $k\ge 3$ and $J=(a_1,\ldots,a_{2k-1})$. Then
        \[
            P(J) = P_{[2k-1]}(J)=[a_1,\ldots,a_{2k-1}].
        \]
    \end{lemma}
    \begin{proof}
        After simple observation, one concludes their minimal non-faces agree. To be precise, observe that the set of minimal non-faces of $[a_1,\ldots,a_{2k-1}]$ is given by
        \[
        \{\phi^{-1}(v_j)\cup\cdots\cup\phi^{-1}(v_{j+k-1})\mid 1\le j\le2k-1\},
        \]
        where the subscripts are mod $2k-1$.
    \end{proof}

    Since the cube has been already covered in the first case, we are enough to check out toric manifolds over $P_{[2k-1]}$. The following theorem is originally proved in \cite{GKB90}, but we include a proof in our language for readers' convenience.

    \begin{theorem} \label{thm:notoricvariety}
        There is no toric manifold over $P_{[2k-1]}$ if $k\ge 4$.
    \end{theorem}
    \begin{proof}
        We are first going to show that there is no toric manifold over $P_{[7]}$. We label the vertices of the heptagon which is a standard Gale diagram of $P_{[7]}$ by $1,2,5,3,4,6,7$ in counterclockwise order. Then the facets $1,2,3,$ and $4$ intersect at a vertex and thus one can assume that the characteristic map of $P_{[7]}$ is given by the matrix
        \[
            \lambda = \begin{pmatrix}
                1   & 0 & 0 & 0 & a & e & i \\
                0 & 1 & 0 & 0 & b & f & j   \\
                0 & 0 & 1 & 0 & c & g & k   \\
                0 & 0 & 0 & 1 & d & h & l
            \end{pmatrix},
        \]
        where $a,b,c,d,e,f,g,h,i,j,k,$ and $l$ are integers. We denote by $\lambda_\alpha$ the $\alpha$-th column vector of $\lambda$.
        \begin{table}
            \[
                \begin{array}{|c|c||c|c|}
                  \hline
                  \text{facet} & \text{sign} & \text{facet} & \text{sign} \\ \hline
                  1234 & + & 1457 & - \\
                  1236 & - & 2347 & + \\
                  1246 & + & 2367 & - \\
                  1347 & - & 2456 & - \\
                  1356 & - & 2457 & + \\
                  1357 & + & 2567 & + \\
                  1456 & + & 3567 & - \\
                  \hline
                \end{array}
            \]
            \caption{Facets of $P^*_{[7]}$ and their signs.}\label{tab:p7vertexsigns}
        \end{table}
        Note that every fan-giving characteristic map is positive in the sense of Definition~\ref{def:positivechr}. Suppose $\alpha\beta\gamma\delta$ and $\alpha\beta\gamma\epsilon$ are two simplices of $P^*_{[7]}$ sharing the triangle $\alpha\beta\gamma$. Then the positiveness of $\lambda$ implies that
        \[
            \det\begin{pmatrix}
            \lambda_\alpha & \lambda_\beta & \lambda_\gamma & \lambda_\delta
            \end{pmatrix}\cdot
            \det\begin{pmatrix}
            \lambda_\alpha & \lambda_\beta & \lambda_\gamma & \lambda_\epsilon
            \end{pmatrix} = -1.
        \]
        Table~\ref{tab:p7vertexsigns} indicates every simplex $\alpha\beta\gamma\delta$ of $P^*_{[7]}$ and its sign $\operatorname{sgn}(\alpha\beta\gamma\delta)$ so that
        $$
            \det\begin{pmatrix}
            \lambda_\alpha & \lambda_\beta & \lambda_\gamma & \lambda_\delta
            \end{pmatrix} = \operatorname{sgn}(\alpha\beta\gamma\delta)\cdot 1.
        $$
        After substituting $g=h=i=j = -1$, we obtain the following system of Diophantine equations:
        \begin{align*}
            b+fd &= 1 \\
            bl+d &= -1 \\
            b+cf &= -1 \\
            bk+c &= -1 \\
            el &= 0 \\
            a+ce &= -1 \\
            ak+c &= -1
        \end{align*}
        \[
            \begin{vmatrix}
                a & e & -1 \\
                c & -1 & k \\
                d & -1 & l
            \end{vmatrix} = -1
        \]
        \[
            \begin{vmatrix}
                a & e & -1 \\
                b & f & -1 \\
                d & -1 & l
            \end{vmatrix} = -1.
        \]
        It is not very hard to show that this has no integral solution, showing that there is no toric manifold over $P_{[7]}$. For $k>4$, apply Lemma~\ref{lem:GaleandVertexfigure} to see that any facet of $P_{[2k-1]}$ is $[2,1,\dotsc,1]$ with $(k-2)$ $1$'s. Since $P_{[2k-3]}$ is a facet of $[2,1,\dotsc,1]$, we conclude that for every $k>4$, $P_{[2k-1]}$ has a face isomorphic to $P_{[7]}$. If $P_{[2k-1]}$ admitted a non-singular fan over it, then its projected fan to $P_{[7]}$ would be non-singular, which is a contradiction.
    \end{proof}

    So far, we have shown that every possible underlying simplicial complex (or, equivalently, the dual of the boundary of a simple polytope) is $\partial P(J)^*$ either $P=[2,2,2]$ or $P=P_{[5]}$. Since we already dealt with the cube $[2,2,2]$, we are remaining with the pentagon.
    \begin{convention}
        When $P$ is a pentagon, its standard Gale diagram is also a pentagon and there is danger of confusion of facet labeling. When we denote $P$ by $P_5$, we assume that its facets are labeled by $1,2,3,4,5 \in \Z_5$ such that two facets $i$ and $j$ intersects if and only if $j-i=\pm 1$.
    \end{convention}
    To apply Proposition~\ref{prop:FanandWedge} for $P(J)$, the first step would be the following lemma.
    \begin{lemma}\label{lem:toricoverpentagon}
        Up to rotational symmetry of $P = P_{5}$ and basis change of $\Z^2$, any complete non-singular fan over $P$ is described by the following characteristic matrix
        \[
            \lambda_d:=\begin{pmatrix}
                1   &   0   &   -1  &   -1  &   d\\
                0   &   1   &   1   &   0   &   -1
            \end{pmatrix}
        \]
        for an arbitrary $d\in \Z$. Suppose each column is numbered by $i+1$, $i+2$, $i+3$, $i+4$ and $i$, from left to right, respectively. We say the corresponding characteristic map (or fan) is of type $(i,d)\in\Z_5\times\Z$ and we write its Davis-Januszkiewicz equivalence class by $(i,d)$. Every class $(i,d)$ is distinct each other except the five cases $(i,0)=(i+1,1)$, $i\in\Z_5$.
    \end{lemma}
    \begin{proof}
        A proof is given by a direct calculation. By basis change and positiveness of the characteristic map, we can assume that the characteristic matrix has the form
        \[
            \lambda = \begin{pmatrix}
                1   &   0   &   -1  &   b  &   d\\
                0   &   1   &   a  &   c   &   -1
            \end{pmatrix}
        \]
        and we have the relations
        \begin{align*}
            -c -ab &= 1 \\
            -b -cd &= 1.
        \end{align*}
        So we obtain
        \[
            c = -ab - 1
        \]
        and
        \[
            d = \frac{1+b}{1+ab} \quad\text{if } ab\ne -1.
        \]
        Consider the following cases:
        \begin{enumerate}
            \item $a=0$. Then
            \[
                \lambda = \begin{pmatrix}
                    1   &   0   &   -1  &   b  &   1+b\\
                    0   &   1   &   0  &   -1   &   -1
                \end{pmatrix}.
            \]
            \item $b=0$. Then
            \[
                \lambda = \begin{pmatrix}
                    1   &   0   &   -1  &   0  &   1 \\
                    0   &   1   &   a  &   -1   &   -1
                \end{pmatrix}.
            \]
            \item $a=1$, $b\ne -1$. Then
            \[
                \lambda = \begin{pmatrix}
                    1   &   0   &   -1  &   b  &   1 \\
                    0   &   1   &   1  &   -1-b   &   -1
                \end{pmatrix}.
            \]
            \item $b=-1$, $a\ne 1$. Then
            \[
                \lambda = \begin{pmatrix}
                    1   &   0   &   -1  &   -1  &   0 \\
                    0   &   1   &   a  &   -1+a   &   -1
                \end{pmatrix}.
            \]
            \item $a=1$, $b=-1$. Then
            \[
                \lambda = \begin{pmatrix}
                    1   &   0   &   -1  &   -1  &   d \\
                    0   &   1   &   1  &   0   &   -1
                \end{pmatrix}.
            \]
        \end{enumerate}
        For these cases, check that $\lambda$ is fan-giving for any $a$, $b$, or $d\in\Z$. It is easy to see these five cases are equivalent up to rotation of $P_5$ and basis change. The fact that the classes $(i,d)$ are distinct each other except $(i,0)=(i+1,1)$, $i\in\Z_5$ is also easily shown.

        For remaining possibility, assume that $a\ne 0,1$ and $b\ne 0,-1$. Since $d$ is an integer, one has the inequality $   |1+b|\ge|1+ab| $. By squaring each hand side, we get
        \[
            b(a-1)\left[(a+1)b + 2\right] \le 0.
        \]
        If $a>1$, the inequality above has no integer solution, and therefore we have $a<0$. Observe that there are five possibilities of $(a,b)$ for $d = (1+b)/(1+ab)$ to be an integer for $a<0$, that is,
        \[
            (a,b)=(-1,2),(-1,3),(-2,1),(-2,2),(-3,1),
        \]
        and $\lambda$ is not fan-giving for any of them.
    \end{proof}

    For a given integral vector $J=(a_1,a_2,\dotsc,a_5)$, $a_i\ge 1$, we know that $P_5(J) = P_5(a_1,a_2,\dotsc,a_5) = [a_1,a_3,a_5,a_2,a_4]$ up to symmetry of the pentagon. First, assume that $a_1 = \dotsb = a_5 = 1$ and we are given a characteristic matrix $\lambda_d = (v_1\; v_2\; v_3\; v_4 \; v_5)$ where $v_i$ is the $i$-th column vector of $\lambda_d$. Suppose we perform a wedge operation on the facet $3$ for example and rename facets by $1,2,3_1,3_2,4$, and $5$. Let $\lambda$ be a characteristic matrix for the wedged polytope $\wed_3 P_5 = P_5(1,1,2,1,1)$ and assume that $\proj_{3_1}\lambda = \lambda_d$. See Case III of the proof of Lemma~\ref{lem:wedge} and one knows that the set of facets $\{i,i+1,3_1\}$, $i\in\Z_5$ corresponds to a vertex of $P_5(1,1,2,1,1)$ whenever $i\le 3$ and $i+1\le 3$. By convention, we choose $i$ so that neither $i$ nor $i+1$ intersects $3$, so in this case $i=5$ and we further choose a basis of $\Z^3$ such that
    \begin{enumerate}
        \item $\lambda(3_1) = (0,0,1)^T$,
        \item $\lambda(i) = (v_i^T \; 0)^T$ and $\lambda(i+1) = (v_{i+1}^T \; 0)^T$.
    \end{enumerate}
    Then the matrix $\lambda$ should look like the following:
    \[
        \lambda=\left(\,\begin{matrix}
            1 & 2 & 3_1 & 3_2 & 4 & 5 \\ \hline
            v_1 & v_2 & 0 & v_3 & v_4 & v_5 \\
             & &        0 &      & & \\ \hline
             0  & n_2 & 1 & n_3 & n_4 & 0
        \end{matrix}\,\right),
    \]
    where $n_j$ is the third entry of the column vector corresponding to the facet $j$ for $j\ne i$. The integer $n_3$ must be $-1$ since
    \[
        \det \left( \lambda(i)\;\lambda(i+1)\;\lambda(3_2)\right) = -\det \left( \lambda(i)\;\lambda(i+1)\;\lambda(3_1)\right) = -1
    \]
    by positiveness of $\lambda$. Let us call $n_{i-1}$ and $n_{i+1}$ be the \emph{unknowns} of the third row. This observation works for any $i\in\Z_5$ and in general one can construct a characteristic matrix for $P(J)$ starting from $\lambda_d$ by repeatedly adding a row and a column. One thing more, note that $\lambda$ is a canonical extension in the sense of \cite{Ewa86} if $n_2 = n_4 = 0$.

    From now on, let us check when $\lambda$ is fan-giving for given $i\in\Z_5$.

    \begin{enumerate}
        \item $i=1$. Then
         \[
            \lambda=\left(\,\begin{matrix}
                0 & 1 & 0 & -1 & -1 & d \\
                0 & 0 & 1 &  1 & 0  & -1 \\
                1 &-1 & n_2 &  0& 0 & n_5
            \end{matrix}\,\right).
        \]
        To compute $\proj_{1_2}\lambda$, add the third row to the first one and delete the second column and third row and one obtains
        \[
            \begin{pmatrix}
                1 & n_2 & -1 & -1 & d+n_5 \\
                0 & 1 & 1 & 0 & -1
            \end{pmatrix}.
        \]
        This characteristic matrix is fan-giving if and only if $n_2 = 0$.
        \item $i=2$. Then
        \[
            \lambda = \begin{pmatrix}
                1 & 0 & 0 & -1 & -1 & d \\
                0 & 0 & 1 & 1 & 0 & -1 \\
                n_1 & 1 & -1 & n_3 & 0 & 0
            \end{pmatrix}
        \]
        and a similar calculation gives $n_3 = 0$ and $dn_1 = 0$.
        \item $i=3$. Then
        \[
            \lambda = \begin{pmatrix}
                1 & 0 & 0 & -1 & -1 & d \\
                0 & 1 & 0 & 1 & 0 & -1 \\
                0 & n_2 & 1 & -1 & n_4 & 0
            \end{pmatrix}.
        \]
        In this case $n_2 = 0$ and $n_4(d-1) = 0$.
        \item $i=4$. Then
        \[
            \lambda = \begin{pmatrix}
                1 & 0 & -1 & 0 & -1 & d \\
                0 & 1 & 1 & 0 & 0 & -1 \\
                0 & 0 & n_3 & 1 & -1 & n_5
            \end{pmatrix}
        \]
        and we obtain $n_3 = 0$.
        \item $i=5$. Then
        \[
            \lambda = \begin{pmatrix}
                1 & 0 & -1 & -1 & 0 & d \\
                0 & 1 & 1 & 0 & 0 & -1 \\
                n_1 & 0 & 0 & n_4 & 1 & -1
            \end{pmatrix}
        \]
        and one has $dn_1 = 0$ and $(d-1)n_4 = 0$.
    \end{enumerate}
    Suppose, for example, that $i=2$ and $n_3 = d = 0$. Then the characteristic matrix $\lambda$ has the form
    \[
        \lambda = \begin{pmatrix}
            1 & 0 & 0 & -1 & -1 & 0 \\
            0 & 0 & 1 & 1 & 0 & -1 \\
            n_1 & 1 & -1 & 0 & 0 & 0
        \end{pmatrix}.
    \]
    Since the D-J classes $(5,0)$ and $(1,1)$ are the same, $\lambda$ is D-J equivalent to
    \[
        \begin{pmatrix}
            1 & 0 & 1 & 0 & -1 & -1 \\
            1 & 0 & 0 & 1 & 1 & 0 \\
            n_1 & 1 & -1 & 0 & 0 & 0
        \end{pmatrix},
    \]
    which is exactly the case $i=1$ with columns re-labeled by $5,1_1,1_2,2,3,4$ and $n_5$ is replaced by $n_1$. In fact, up to rotation of $P$, we do not need to consider the cases $n_j \ne 0$ for $j\ne 5$.

    The next step is to deal with non-singular characteristic maps over $P_5(a_1,\dotsc,a_5)$ when $a_1+\dotsb+a_5 = 7$. Each of them corresponds to a twice wedged pentagon. This time, let us omit canonical extensions and assume every nonzero unknown of each row lies in the column $5$. Then there are three cases:
    \begin{enumerate}
        \item wedged twice at 4;
        \item wedged twice at 1;
        \item and wedged 4 and 1.
    \end{enumerate}
    When wedged twice at 4, the matrix has the form
    \[
        \lambda = \begin{pmatrix}
            1 & 2 & 3 & 4_1 & 4_2 & 4_3 & 5 \\ \hline
            1 & 0 & -1 & 0 & 0 & -1 & d \\
            0 & 1 & 1 & 0 & 0  & 0 & -1 \\
            0 & 0 & 0 & 1 & 0 & -1 & n_5 \\
            0 & 0& 0 & 0 & 1 & -1 & m_5
        \end{pmatrix}.
    \]
    Reminding Corollary~\ref{cor:projectionofKJ}, note that this matrix has three possible projections over the pentagon: $\proj_{\{4_1,4_2\}}\lambda,\,\proj_{\{4_1,4_3\}}\lambda$, and $\proj_{\{4_2,4_3\}}\lambda$. Among these, $\proj_{\{4_1,4_2\}}\lambda = \lambda_d$ and the other two are also fan-giving. The case wedged twice at 1 is similarly done. For the case wedged 1 and 4, we write down the matrix
    \[
        \lambda = \begin{pmatrix}
            1_1& 1_2 & 2 & 3 & 4_1 & 4_2 & 5 \\ \hline
            0 & 1 & 0 & -1 & 0  & -1 & d \\
            0 & 0 & 1 & 1 & 0  & 0 & -1 \\
            0 & 0 & 0 & 0 & 1 & -1 & n_5 \\
            1 & -1 & 0 & 0 & 0 & 0 & m_5
        \end{pmatrix}
    \]
    and consider its four projections which are characteristic map for the pentagon. Here, let us compute $\proj_{\{1_2,4_2\}}\lambda$ skipping the other easier three. By adding the first row to the fourth row to get
    \[
        \begin{pmatrix}
            1_1& 1_2 & 2 & 3 & 4_1 & 4_2 & 5 \\ \hline
            0 & 1 & 0 & -1 & 0  & -1 & d \\
            0 & 0 & 1 & 1 & 0  & 0 & -1 \\
            0 & 0 & 0 & 0 & 1 & -1 & n_5 \\
            1 & 0 & 0 & -1 & 0 & -1 & m_5+d
        \end{pmatrix}.
    \]
    The following matrix $\proj_{1_2}\lambda$ is obtained by deleting the first row and the column $1_2$:
    \[
        \proj_{1_2}\lambda = \begin{pmatrix}
            1_1 & 2 & 3 & 4_1 & 4_2 & 5 \\ \hline
            0   & 1 & 1 & 0  & 0 & -1 \\
            0   & 0 & 0 & 1 & -1 & n_5 \\
            1   & 0 & -1 & 0 & -1 & m_5+d
        \end{pmatrix}.
    \]
    Now projecting it with respect to $4_2$ gives
    \[
        \proj_{\{1_2,4_2\}}\lambda = \begin{pmatrix}
            0 & 1 & 1 & 0 & -1 \\
            1 & 0 & -1 & -1 & m_5+d-n_5
        \end{pmatrix},
    \]
    which is fan-like for all $m_5,d,n_5$ and has type $(5,d+m_5-n_5)$.

    Now we are ready to deal with general $P_5(a_1,\dotsc,a_5)$. Up to rotation of $P$, we can assume that every nonzero unknowns lie in the column $5_j$ for some $1\le j\le a_5$. We start with $P_5$ and perform wedges at 1 $a_1-1$ times and continue wedging at 2 $a_2-1$ times and so on. In other words, we do the row-and-column adding in the order 1, 2, 3, 4, and 5.
    For convenience of notation, we write
    \[
        M_i := \begin{pmatrix}
              0 & \cdots & 0 & v_i\\
        \end{pmatrix}_{2\times a_i}
    \]
    where $i=1,\dotsc,5$ and $v_i$ is the $i$-th column of $\lambda_d$. Moreover, we write
    \[
        S_{a_i}:=\begin{pmatrix}
              1 &    &   0 &  -1  \\
                & \ddots  &    & \vdots   \\
               0   &   & 1 & -1
        \end{pmatrix}_{(a_i-1)\times a_i}
    \]
    and
    \[
        N_i := \begin{pmatrix}
             0 & \cdots & 0 & n_i \\
        \end{pmatrix}_{(a_i-1)\times a_5}
    \]
    for an arbitrary integral vector $n_i=(n_{i2},\dotsc,n_{ia_i})^T$. We put $n_{i1}=0$ for convention. Any pentagon in the 2-skeleton of $P(J)=P_5(a_1,\dotsc,a_5)$ can be labeled by an integral vector $(b_1,\dotsc,b_5)$, $1\le b_i\le a_i$, and each pentagon can be naturally identified with $P$.

    \begin{theorem}\label{thm:toricoverpentagon}
        Up to rotational symmetry of the pentagon $P_5$ and basis change, a toric manifold over $P_5(a_1,\dotsc,a_5)$ is determined by the following characteristic matrix
        \[
            \Lambda=\begin{pmatrix}
                M_1 & M_2 & M_3 & M_4 & M_5 \\
                S_{a_1} & 0 & 0 & 0 & N_1 \\
                0 & S_{a_2} & 0 & 0 & 0 \\
                0 & 0 & S_{a_3} & 0 & 0 \\
                0 & 0 & 0 & S_{a_4} & N_4 \\
                0 & 0 & 0 & 0 & S_{a_5}
            \end{pmatrix}_{(\sum_i{a_i}-3)\times\sum_i{a_i}},
        \]
        for arbitrary choice of $d$, $n_1$, and $n_4$. Further, the projection of $\Lambda$ over the pentagon labeled by $(b_1,\dotsc,b_5)$ has type $(5,d+n_{1b_1}-n_{4b_4})$.
    \end{theorem}
    \begin{remark}\label{rem:existenceandclassificationoftoricandtopoltoric}
        For a standard Gale diagram $P_{2k-1}$ for a simple polytope $P$, we know that there is a toric manifold over $P$ if and only if $k=2$ and $3$. A result in \cite{Ero08} states that there is a topological toric manifold (or, equivalently, a quasitoric manifold) over $P$ if and only if $k=2,3,4$. We have classified toric manifolds of Picard number 3, but the classification of topological toric manifolds over $K$ when $\Pic(K)=3$ is somewhat complicated to calculate to be contained here. It will be covered elsewhere in the future.
    \end{remark}

\section{Application: Projectivity of toric varieties}\label{sec:projectivity}
    By Theorem~\ref{thm:mainthm}, there is some kind of good relationship between toric objects (i.e., topological toric manifolds, quasitoric manifolds, and toric manifolds) and wedges of simplicial complexes. It is also true for the category of toric orbifolds or the category of complete (not necessarily non-singular) rational fans. We can show this is not true for the category of projective toric orbifolds or that of complete strongly polytopal rational fans using Proposition~\ref{prop:ShephardandWedge}. To be more precise, there exists a complete non-strongly polytopal fan $\Sigma$ over $\wed_v(K)$ whose projections $\proj_{v_1}(\Sigma)$ and $\proj_{v_2}(\Sigma)$ are strongly polytopal.

    \begin{example}
        Define the characteristic map $(\wed_1P_{[7]},\;\lambda)$ by the matrix
        $$
            \lambda = \begin{pmatrix*}[r]
                1_1 & 1_2 & 2 & 3 & 4 & 5 & 6 & 7 \\ \hline
                -16 & 16 & -1 & 0 & 0 & 0 & 0 & 1 \\
                -33 & 83 & -6 & 0 & 0 & 0 & 1 & 0 \\
                -37 & 127 & -10 & 0 & 0 & 1 & 0 & 0 \\
                -33 & 123 & -10 & 0 & 1 & 0 & 0 & 0 \\
                -13 & 63 & -6 & 1 & 0 & 0 & 0 & 0
            \end{pmatrix*}
        $$
        which is fan-giving \footnote{This can be easily shown using a computer program such as the Maple package \emph{Convex}\cite{Fra_Conv}.} and hence defines a complete fan $\Sigma$. To compute its Shephard diagram, we multiply $10$ to last 6 columns respectively, obtaining
        \[
            X = \begin{pmatrix*}[r]
                1_1 & 1_2 & 2 & 3 & 4 & 5 & 6 & 7 \\ \hline
                -16 & 16 & -10 & 0 & 0 & 0 & 0 & 10 \\
                -33 & 83 & -60 & 0 & 0 & 0 & 10 & 0 \\
                -37 & 127 & -100 & 0 & 0 & 10 & 0 & 0 \\
                -33 & 123 & -100 & 0 & 10 & 0 & 0 & 0 \\
                -13 & 63 & -60 & 10 & 0 & 0 & 0 & 0
            \end{pmatrix*}
        \]
        the sum of whose column is zero and therefore a Shephard diagram for $\Sigma$ can be computed by the matrix
        \[
            \begin{pmatrix*}[r]
                \widehat{1}_1 & \widehat{1}_2 & \widehat{2} & \widehat{3 }& \widehat{4} & \widehat{5} & \widehat{6} & \widehat{7} \\ \hline
                 -2.5& 2.5& 4& 5& 1& -1& -5& -4 \\
                 5.5& 5.5& 5& 2.5& 0.5& 0.5& 2.5& 5
            \end{pmatrix*}.
        \]
    \begin{figure}[h]
        \begin{tikzpicture}[scale=0.9, ultra thick]
            \coordinate [label=above:{$\widehat{1_1}$}    ](0) at (-2.4,5.5);
            \coordinate [label=above:{$\widehat{1_2}$}    ](1) at (2.4,5.5);
            \coordinate [label=above right:{$\widehat{2}$}](2) at (4,5);
            \coordinate [label=right      :{$\widehat{3}$}](3) at (5,2.5);
            \coordinate [label=below right:{$\widehat{4}$}](4) at (1,.5);
            \coordinate [label=below left :{$\widehat{5}$}](5) at (-1,.5);
            \coordinate [label=left       :{$\widehat{6}$}](6) at (-5,2.5);
            \coordinate [label=above left :{$\widehat{7}$}](7) at (-4,5);
            \begin{scope}[fill=red]
                \clip (5)--(0)--(4)--cycle;
                \clip (2)--(6)--(3)--cycle;
                \fill (4)--(7)--(3)--cycle;
            \end{scope}
            \begin{scope}[fill=blue]
                \clip (5)--(1)--(4)--cycle;
                \clip (6)--(3)--(7)--cycle;
                \fill (5)--(2)--(6)--cycle;
            \end{scope}
            \draw (4)--(7)--(3)--(6)--(2)--(5)
                (5)--(0)--(4)--cycle
                (5)--(1)--(4)--cycle;
        \end{tikzpicture}
        \caption{A Shephard diagram for a complete non-strongly polytopal fan over $\wed_1P_{[7]} = [2,1,1,1,1,1,1]$.}
        \label{fig:ShephardofNonpolytopal}
    \end{figure}
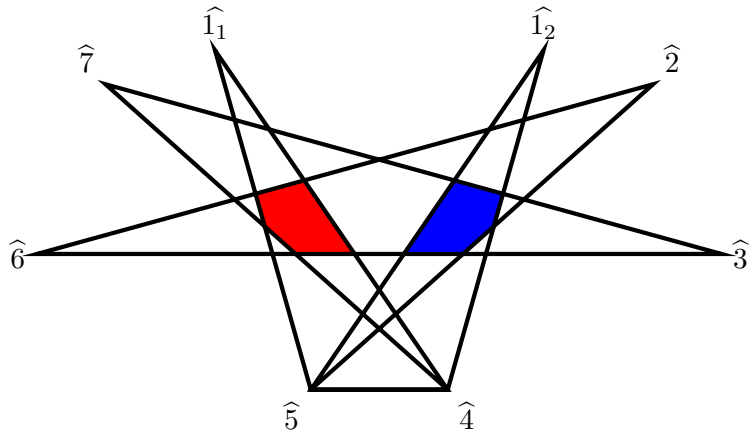
        An illustration for $\widehat{X}$ is given in Figure~\ref{fig:ShephardofNonpolytopal}.
        Observe that, by Proposition~\ref{prop:ShephardandWedge}, $S(\proj_{1_1}\Sigma, \widehat{X}\setminus(\widehat{1}_{ 1}))$ is the blue region and $S(\proj_{1_2}\Sigma, \widehat{X}\setminus(\widehat{1}_{ 2}))$ is the red one which are nonzero respectively, so each projected fan is strongly polytopal. But $S(\Sigma,\widehat{X})$ is the intersection of the two sets and is empty and hence $\Sigma$ itself is not strongly polytopal. Note that there is no complete non-singular fan over $P_{[7]}$ by Theorem~\ref{thm:notoricvariety}. Therefore, $\Sigma$ is not non-singular. The authors do not know whether such an example of a non-singular fan exists or not.
    \end{example}

    \begin{question}
        Is there a complete non-singular non-projective fan $\Sigma$ over $\wed_{v}(K)$ whose projected fans $\proj_{v_{1}}\Sigma$ and $\proj_{v_{2}} \Sigma$ are projective?
    \end{question}
    In general, as one has seen, projectivity of projected fans of $\Sigma$ with respect to $v_{1}$ and $v_{2}$ does not guarantee projectivity of $\Sigma$ over  $\wed_{v}(K)$ . But in special cases of fans, one can prove their projectivity as we will see in the rest of this section.

    Let $(K,\lambda)$ be a fan-giving complete non-singular characteristic map of dimension $n$ and $M = M(K,\lambda)$ be the corresponding toric manifold. If $K$ has $m$ vertices, then the number $m-n$ is known as the \emph{Picard number} of the toric variety $M$. Since we have classified smooth toric varieties of Picard number $3$, we can try checking their projectivity using Proposition~\ref{prop:ShephardandWedge} and prove the following:

    \begin{theorem}\label{thm:picnum3isprojective}
        Every toric manifold of Picard number three is projective.
    \end{theorem}
    The above theorem was originally proved by Kleinschmidt and Sturmfels \cite{KS91}, but their proof was  somewhat cumbersome case-by-case approach. Here we present a new proof.

    We can assume that $K = \partial P(J)^*$ where $P$ is either a cube $I^3$ or a pentagon $P_{[5]}$. First, let us consider when $P$ is a cube. That is, let $\Sigma$ be a complete (not necessarily non-singular) fan over $K = \star_{i=1}^3\partial\Delta^{n_i} = \partial\Delta^{n_1}\star\partial\Delta^{n_2}\star\partial\Delta^{n_3}$. We assume that every projected fan of $\Sigma$ over $\partial(I^3)^*$ is strongly polytopal. Label the vertices of $K$ by $1_0,1_1,1_2,\dotsc,1_{n_1},2_0,\dotsc,2_{n_2},3_0,\dotsc,3_{n_3}$ and let
    \[
        \widehat{X} = (\widehat{1}_0, \widehat{1}_1,\dotsc,\widehat{1}_{n_1},\widehat{2}_0, \dotsc, \widehat{2}_{n_2}, \widehat{3}_0, \dotsc, \widehat{2}_{n_3})
    \]
    be a Shephard diagram for $\Sigma$. Choose two sequences $\mathbf{a}=\{a_i\}$ and $ \mathbf{b}=\{b_i\}$, $i=1,2,3$, such that $0\le a_i < b_i \le n_i$ for all $i$. For such a choice, one has the corresponding projected fan over $\partial(I^3)^*$, denoted by $\Sigma_{\mathbf{ab}}$, which is determined by the vertex set $\{1_{a_1},1_{b_1},2_{a_2},2_{b_2},3_{a_3},3_{b_3}\}$. By Proposition~\ref{prop:Shephardandproj}, the subsequence
    \[
        \widehat{X}_{\mathbf{ab}} := (\widehat{1}_{a_1},\widehat{1}_{b_1}, \widehat{2}_{a_2},\widehat{2}_{b_2}, \widehat{3}_{a_3},\widehat{3}_{b_3})
    \]
    is a Shephard diagram of the projected fan and therefore the set
    \[
        S(\Sigma_{\mathbf{ab}},\widehat{X}_{\mathbf{ab}}) = \bigcap_{k_i=a_i\text{ or }b_i}\relint\conv(\widehat{1}_{k_1},\widehat{2}_{k_2}, \widehat{3}_{k_3})
    \]
    is nonempty. By Corollary~\ref{cor:galeforpolytopes}, every such $\conv(\widehat{1}_{k_1},\widehat{2}_{k_2}, \widehat{3}_{k_3})$ is a triangle. For simplicity of notation, let us temporarily write $\relint \conv \{v_1,\dotsc,v_m\} = v_1\cdots v_m$. For example, since
    \[
        \widehat{1}_{k_1} \widehat{2}_{k_2} \widehat{3}_{a_3}  \cap
        \widehat{1}_{k_1} \widehat{2}_{k_2} \widehat{3}_{b_3}   \ne \varnothing
    \]
    for every $k_1$ and $k_2$, every point $\widehat{3}_{j}$ lies on the same side of the line $\overline{\widehat{1}_{k_1} \widehat{2}_{k_2}}$ for any $k_1$ and $k_2$. Let us write $C_i := \conv \{i_{j}\mid 0\le j \le n_i\}$. Then actually one can find $d_1$ and $d_2$, $0\le d_1\le n_1$, $0\le d_2\le n_2$ such that the line $\overline{\widehat{1}_{d_1} \widehat{2}_{d_2}}$ divides the sets $C_3$ and $C_1 \cup C_2$ where $C_1\cup C_2$ is allowed to intersect $\overline{\widehat{1}_{d_1} \widehat{2}_{d_2}}$. Similarly, we can choose $e_2$, $e_3$, $f_3$, and $f_1$, such that the line $\overline{\widehat{2}_{e_2} \widehat{3}_{e_3}}$ divides $C_1$ and $C_2\cup C_3$ and the line $\overline{\widehat{3}_{f_3} \widehat{1}_{f_1}}$ divides $C_2$ and $C_3\cup C_1$. Then by projectivity assumption, the set
    \begin{multline*}
        Z:=\widehat{1}_{d_1} \widehat{2}_{d_2} \widehat{3}_{e_3} \cap \widehat{1}_{d_1} \widehat{2}_{d_2} \widehat{3}_{f_3} \cap
        \widehat{1}_{d_1} \widehat{2}_{e_2} \widehat{3}_{e_3} \cap \widehat{1}_{d_1} \widehat{2}_{e_2} \widehat{3}_{f_3}  \\
        \cap
        \widehat{1}_{f_1} \widehat{2}_{d_2} \widehat{3}_{e_3} \cap  \widehat{1}_{f_1} \widehat{2}_{d_2} \widehat{3}_{f_3} \cap
        \widehat{1}_{f_1} \widehat{2}_{e_2} \widehat{3}_{e_3} \cap \widehat{1}_{f_1} \widehat{2}_{e_2} \widehat{3}_{f_3}
    \end{multline*}
    is nonempty. For any triangle $ \widehat{1}_{k_1}\widehat{2}_{k_2}\widehat{3}_{k_3}$, its edges do not intersect $Z$ and we conclude that $Z\subseteq \widehat{1}_{k_1}\widehat{2}_{k_2}\widehat{3}_{k_3}$ proving that $\Sigma$ is a strongly polytopal fan.
    We remark two things. First, if $\Sigma$ is non-singular, then its corresponding toric variety is a generalized Bott manifold of stage three. Secondly, the argument above can be generalized to fans whose underlying complex is $K = \star_{i=1}^k\partial\Delta^{n_i}$ by replacing dividing lines for dividing hyperplanes. We state the result as a proposition.
    \begin{proposition}
        Let $\Sigma$ be a complete fan over the simplicial sphere $K = \star_{i=1}^k\partial\Delta^{n_i}$ with vertices
        \[
            1_0,1_1,\dotsc,1_{n_1},2_0,\dotsc,2_{n_2},\dotsc,k_0,\dotsc,k_{n_k}.
        \]
        For $i=1,\dotsc,k$. choose
        Assume that, for every sequence of integers $c$ such that $0\le c_{i1}\le c_{i2}\le\dotsb\le c_{i,{k-1}}\le n_i$, the projected fan of $\Sigma$ with 1-cones given by
        \[
            1_{c_{11}},1_{c_{12}},\dotsc,1_{c_{1,k-1}},\dotsc,  k_{c_{k1}},k_{c_{k2}},\dotsc,1_{c_{k,k-1}}
        \]
        is strongly polytopal. Then $\Sigma$ is strongly polytopal.
    \end{proposition}


    Next, we consider smooth toric varieties over $P(J)$ when $P = P_5$ and $J=(a_1,\dotsc,a_5)$. Note the facets of $P_5(J)$ are $1_1,\dots,1_{a_1},2_1,\dotsc,2_{a_2}, \dotsc, 5_1,\dotsc,5_{a_5}$. Let $M = M(P_5(J),\Lambda)$ where $\Lambda$ is the matrix seen from Theorem~\ref{thm:toricoverpentagon}. Let $\Sigma$ be the fan given by $\Lambda$ and $\widehat{X}$ be a Shephard diagram for $\Sigma$. By a property of Shephard diagrams of canonical extensions, we know that $\widehat{2}_1=\widehat{2}_2=\dotsb=\widehat{2}_{a_2}$, $\widehat{3}_1=\widehat{3}_2=\dotsb=\widehat{3}_{a_3}$, and $\widehat{5}_1=\widehat{5}_2=\dotsb=\widehat{5}_{a_5}$, thus it will be natural that we denote them by just $\widehat{2}$, $\widehat{3}$, and $\widehat{5}$ respectively. Choose a sequence of integers $\mathbf{i}=(i_1, i_4)$ so that $1\le i_1\le a_1$ and $1\le i_4 \le a_4$. Then every subsequence $\widehat{X}_{\mathbf{i}} := (\widehat{1}_{i_1},\widehat{2},\widehat{3},\widehat{4}_{i_4},\widehat{5})$ is a Shephard diagram for a fan $\Sigma'$ over $P$ which is always strongly polytopal. We know the fan $\Sigma'$ is of type $(5,d)$ for some integer $d$. Recall that there is the corresponding characteristic map of $\Sigma'$ (up to basis change of $\Z^2$)
    \[
        \lambda_d=\begin{pmatrix}
            1   &   0   &   -1  &   -1  &   d\\
            0   &   1   &   1   &   0   &   -1
        \end{pmatrix}
    \]
    and we compute a Shephard diagram for $\lambda_d$. Suppose that $d\ge 0$. To make the sum of column vectors zero, we multiply a suitable positive real number to each column of $\lambda_d$, resulting
    \[
        X=\begin{pmatrix}
            2   &   0   &   -1  &   -2d-1  &   2d\\
            0   &   1   &   1   &   0   &   -2
        \end{pmatrix}
    \]
    and we choose a linear transform $\overline{X}$ which contains a row $(1,1,1,1,1)$, for example
    \[
        \overline{X}=\begin{pmatrix}
            1 & -2 & 2 & 0 & 0 \\
            -d & 2 & 0 & 0 & 1 \\
            1 & 1 & 1 & 1 & 1
        \end{pmatrix}.
    \]
    Deleting the row $(1,1,1,1,1)$ gives the wanted Shephard diagram. Note that for all $d \ge 0$, the fifth column vector $(0,1)$ is the midpoint of the second column $(-2,2)$ and the third one $(2,0)$. This is also true for any other Shephard diagrams and in particular one obtains $\widehat{5} = (\widehat{2}+\widehat{3})/2$. A similar argument works for $d<0$.

    We use again the notation $\relint \conv \{v_1,\dotsc,v_m\} = v_1\cdots v_m$. We know that for every $\mathbf{i}$,
    \[
        S(\widehat{X}_{\mathbf{i}}) := \widehat{1}_{i_1}\widehat{2}\widehat{3} \cap \widehat{2}\widehat{3} \widehat{4}_{i_4} \cap \widehat{3} \widehat{4}_{i_4}\widehat{5} \cap \widehat{4}_{i_4}\widehat{5} \widehat{1}_{i_1} \cap \widehat{5} \widehat{1}_{i_1}\widehat{2} \ne \varnothing
    \]
    by strong polytopalness of $\Sigma'$. Since $\widehat{1}_{i_1}\widehat{2}\widehat{3}$ intersects $\widehat{2}\widehat{3} \widehat{4}_{i_4}$ for every $\mathbf{i}$, every $\widehat{1}_{i_1}$ and $\widehat{4}_{i_4}$ must lie on the same open half-plane determined by the line $\overline{\widehat{2}\widehat{3}}$. The fact that $\widehat{3} \widehat{4}_{i_4}\widehat{5} \cap \widehat{5} \widehat{1}_{i_1}\widehat{2}$ is nonempty implies that $\angle \widehat{1}_{i_1}\widehat{5}\widehat{3} + \angle \widehat{4}_{i_4}\widehat{5}\widehat{2} < 2\pi$. See Figure~\ref{fig:Shephardofpentagon}.


    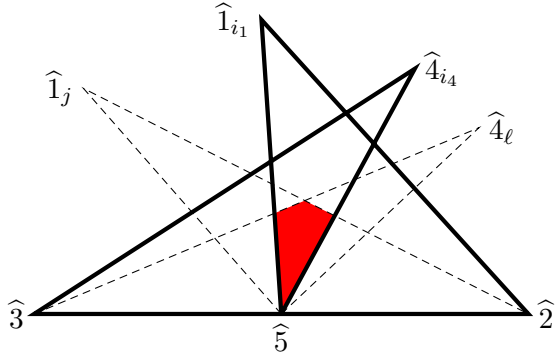
\begin{figure}[h]
        \begin{tikzpicture}[scale=1.3]
            \coordinate [label=left: $\widehat{1}_{i_1}$   ] (1)  at (-0.2,3.5);
            \coordinate [label=left: $\widehat{1}_{j}$     ] (11) at (-2,2.8);
            \coordinate [label=right:$\widehat{2}$         ] (2)  at (2.5,.5);
            \coordinate [label=left: $\widehat{3}$         ] (3)  at (-2.5,.5);
            \coordinate [label=right:$\widehat{4}_{i_4}$   ] (4)  at (1.35,3);
            \coordinate [label=right:$\widehat{4}_{\ell}$  ] (44) at (2,2.4);
            \coordinate [label=below:$\widehat{5}$         ] (5)  at (0,.5);
            \begin{scope}[fill=red]
                \clip (3)--(44)--(5)--cycle;
                \clip (5)--(11)--(2)--cycle;
                \fill (1)--(5)--(4)--cycle;
            \end{scope}
            \draw [ultra thick] (5)--(1)--(2)--(3)--(4)--(5);
            \draw [thin, densely dashed] (5)--(11)--(2)
                (5)--(44)--(3);
        \end{tikzpicture}
        \caption{A Shephard diagram for a complete non-singular fan $\Sigma$ over $P_5( 2,1,1,2,1) $. The thick lines indicate a Shephard diagram $\widehat{X}_{\mathbf{i}}$ which has maximal angles of $\angle \widehat{1}_{i_1}\widehat{5}\widehat{3}$ and $\angle \widehat{4}_{i_4}\widehat{5}\widehat{2}$ respectively.}
        \label{fig:Shephardofpentagon}
    \end{figure}
    Pick $i_1$ and $i_4$ such that the angles $\angle \widehat{1}_{i_1}\widehat{5}\widehat{3}$ and $\angle \widehat{4}_{i_4}\widehat{5}\widehat{2}$ are maximal, respectively. Then for such $i_1$ and $i_4$, it is an easy task to show that the intersection of $\relint\conv\{\widehat{1}_{i_1}, \widehat{5}, \widehat{4}_{i_4}\}$ and an $\varepsilon$-ball centered at $\widehat{5}$ is included in $\relint\conv {\widehat{X}|_Y}$ for any coface $Y$ of $P_5(J)$, hence $\Sigma$ is a strongly polytopal fan. For example,  the set $S(\Sigma,\widehat{X})$ is the region colored by red in Figure~\ref{fig:Shephardofpentagon}.

\section{Application: Real toric varieties and their topological analogues}\label{sec:smallcover}
    In this section, we briefly give an introduction to ``real toric objects'' and classify them for special cases. Let $M$ be a toric variety of complex dimension $n$. Then there is a canonical involution, called the \emph{conjugation} of $M$. The set of its fixed points, denoted by $M_{\R}$, is a real subvariety of dimension $n$, called a \emph{real toric variety}. When $M$ is a toric manifold, then $M_{\R}$ is a submanifold of dimension $n$ and called a \emph{real toric manifold}. This concept can be generalized to topological toric case.

    \begin{definition}(\cite{IFM12})
        We say that a closed smooth manifold $M$ of dimension $n$ with an effective smooth action of $(\R^*)^n$ having an open dense orbit is a \emph{real topological toric manifold} if it is covered by finitely many invariant open subsets each of which is equivariantly deffeomorphic to a direct sum of real one-dimensional smooth representation spaces of $(\R^*)^n$.
    \end{definition}
    See \cite{IFM12} for details. Note that $(\R^*)^n\cong \R^n\times\Z_2^n$ as a group. Similarly to its complex counterpart, we can consider the combinatorial object for a real topological toric manifold (as a $\Z_2^n$-manifold). One can define an analogue of a characteristic map, called a \emph{ characteristic map over $\Z_2$}, $(K,\lambda)$: all is the same but $\lambda$ maps to $\Z_2^n$. The non-singularity condition becomes that for each face of $K$, the vectors $\lambda(i)$ are linearly independent over $\Z_2$.

    If $(K,\lambda)$ is polytopal, then the corresponding manifold $M(\lambda)$ is called a \emph{small cover}, which is a $\Z_2$-version of a quasitoric manifold.

    One can observe that a slightly modified version of Proposition~\ref{prop:chrwedge} works for real topological toric manifolds except positiveness of orientation which is not applicable for a $\Z_2$-version. Hence, we have the following theorem which is a $\Z_2$-version of our main theorem restating Theorem~\ref{thm:mainthm2};
    \begin{theorem}
        Let $K$ be a fan-like simplicial sphere and $v$ a given vertex of $K$. Let $(\wed_v(K), \lambda)$ be a characteristic map over $\Z_2$ and let $v_1$ and $v_2$ be the two new vertices of $\wed_v(K)$ created from the wedging. Then $\lambda$ is uniquely determined by the projections $\proj_{v_1}\lambda$ and $\proj_{v_2}\lambda$. Furthermore, $\lambda$ is non-singular if and only if  so are $\proj_{v_1}\lambda$ and $\proj_{v_2}\lambda$.
    \end{theorem}

    Using the theorem, we can classify every real topological toric manifold over $K$ with $\Pic(K) = 3$. The process is quite similar to classification of toric manifolds of Picard number three introduced in the previous section. First, remember that $K$ is polytopal by Mani \cite{Man72} and $K = \partial P^*$ for a simple polytope $P$. Let $X'$ be a standard Gale diagram for $K$ which is a regular $(2k-1)$-gon. Recall Remark~\ref{rem:existenceandclassificationoftoricandtopoltoric} which states there is a topological toric manifold over $P$ if and only if $k=2,3,4$. Similarly, again by \cite{Ero08}, there is a real topological toric manifold (or small cover) over $P$ if and only if $k=2,3,4$.
    Suppose that $k=2$. Then the Gale diagram is a triangle and $P=[n_1+1,n_2+1,n_3+1]$ for positive integers $n_i$. Let $P=[2,2,2]$ be a 3-cube. Every small cover on $P$ is a 3-stage real Bott manifold by \cite{CMS10} and hence for some appropriate ordering of facets of $P$, its characteristic matrix
    \[
    \lambda = \left(
        \begin{array}{ccc|ccc}
        1 & 0 & 0 & 1 & 0 & 0 \\
        0 & 1 & 0 & * & 1 & 0 \\
        0 & 0 & 1 & * & * & 1
        \end{array}
        \right),
    \]
    where the asterisks mean arbitrary numbers in $\Z_2$. Applying the $\Z_2$-version of Proposition~\ref{prop:chrwedge} repeatedly, we obtain the characteristic matrix of a small cover on $[n_1+1,n_2+1,n_3+1]=\Delta^{n_1} \times \Delta^{n_2} \times \Delta^{n_3}$:
    \[
    \left(
        \begin{array}{ccc|ccc|ccc|ccc}
          & & & & & & & & & 1 & 0 & 0 \\
          & I_{n_1} & & & 0 & & & 0 & & \vdots & \vdots & \vdots       \\
          & & & & & & & & & 1 & 0 & 0 \\ \hline
          & & & & & & & & & * & 1 & 0 \\
          &0 & & &I_{n_2} & & & 0& & \vdots & \vdots & \vdots \\
          & & & & & & & & & * & 1 & 0 \\ \hline
          & & & & & & & & & * & * & 1 \\
          &0 & & &0 & & & I_{n_3} & & \vdots & \vdots & \vdots \\
          & & & & & & & & & * & * & 1
        \end{array}
    \right)
    \]
    which is exactly that of a generalized Bott manifold over $\R$ of stage 3 as seen from \cite{CMS10}. In particular, $M$ is a real toric variety. We remark that the number of D-J classes over $\Delta^{n_1} \times \Delta^{n_2} \times \Delta^{n_3}$ is calculated in \cite{Choi08} by
    \begin{multline*}
        \# DJ = 1 + 2(x_1 + x_2 + x_3) + (x_1+x_2+x_3)^2+ (x_1 x_2 + x_2 x_3 + x_3 x_1)\\
         +  (x_1 + x_2 + x_3)(x_1^2 + x_2 ^2 + x_3 ^2) - x_1^3- x_2^3- x_3^3,
    \end{multline*}
    where $x_i = 2^{n_i}-1$ for $i=1,2,3$.

    If $k=3$, the polygon is a pentagon $P_5$. Let us find small covers over $P_5$. To do this, denote $\lambda(i) = v_i$. By applying a basis change, we can assume $v_1=\binom 10$ and $v_2=\binom 01$.  Then it is a simple computation to see that there are five D-J classes on the pentagon:
    \[
        A_1 := \begin{pmatrix}
        1 & 0 & 1 & 0 & 1 \\
        0 & 1 & 1 & 1 & 1
        \end{pmatrix},
    \]
    \[
        A_2 := \begin{pmatrix}
        1 & 0 & 1 & 1 & 1 \\
        0 & 1 & 1 & 0 & 1
        \end{pmatrix},
    \]
    \[
        A_3 := \begin{pmatrix}
        1 & 0 & 1 & 1 & 0 \\
        0 & 1 & 1 & 0 & 1
        \end{pmatrix},
    \]
    \[
        A_4 := \begin{pmatrix}
        1 & 0 & 1 & 1 & 0 \\
        0 & 1 & 0 & 1 & 1
        \end{pmatrix},
    \]
    and
    \[
        A_5 := \begin{pmatrix}
        1 & 0 & 1 & 0 & 1 \\
        0 & 1 & 0 & 1 & 1
        \end{pmatrix}.
    \]
    Note that in any case, up to a suitable rotation of $P_5$, the matrix has the form $(a\;b\; a\; b \; c)$ for some nonzero vectors $a,b,c\in\Z_2^2$ which are distinct each other. Up to basis change, the characteristic map is determined by the position of $c$, hence the name $A_i$. Let us call the position of $c$ the \emph{type} of $A_i$. Hence the type of $A_i$ is $i$.
    The five matrices are all equivalent each other up to rotational symmetry of  $P_{5}$  and the integral matrix $\lambda_d$ becomes $A_3$ mod 2 if $d$ is even or $A_2$ otherwise. Hence we obtain a $\Z_2$-analogue of Lemma~\ref{lem:toricoverpentagon}. Since every small cover over  $P_{5}$  is a real toric variety, it concludes that every small cover over $P_{5}( a_1,a_2,\dotsc, a_5)$ is a real toric variety and all argument used to prove Theorem~\ref{thm:toricoverpentagon} works almost the same. Let $A=A_3$ or $A_2$. Again, we write
    \[
        M_i := \begin{pmatrix}
              0 & \cdots & 0 & v_i\\
        \end{pmatrix}_{2\times a_i}
    \]
    where $i=1,\dotsc,5$ and $v_i$ is the $i$-th column of $A$. Moreover, we write
    \[
        S_{a_i}:=\begin{pmatrix}
              1 &    &   0 &  1  \\
                & \ddots  &    & \vdots   \\
               0   &   & 1 & 1
        \end{pmatrix}_{(a_i-1)\times a_i}
    \]
    and
    \[
        N_i := \begin{pmatrix}
             0 & \cdots & 0 & n_i \\
        \end{pmatrix}_{(a_i-1)\times a_5}
    \]
    for an arbitrary vector $n_i=(n_{i2},\dotsc,n_{ia_i})^T$. We put $n_{i1}=0$ for convention.  Any pentagon in the 2-skeleton of $P(J)=P_5(a_1,\dotsc,a_5)$ can be labeled by an integral vector $(b_1,\dotsc,b_5)$, $1\le b_i\le a_i$, and each pentagon can be naturally identified with $P_5$. Now we get the following $\Z_2$-analogue of Theorem~\ref{thm:toricoverpentagon}.

    \begin{theorem} \label{thm:realtoricoverpentagon}
        Up to rotational symmetry of the pentagon $P_5$ and basis change, a small cover over $P_5(a_1,\dotsc,a_5)$ is determined by the following characteristic matrix
        \[
            \Lambda=\begin{pmatrix}
                M_1 & M_2 & M_3 & M_4 & M_5 \\
                S_{a_1} & 0 & 0 & 0 & N_1 \\
                0 & S_{a_2} & 0 & 0 & 0 \\
                0 & 0 & S_{a_3} & 0 & 0 \\
                0 & 0 & 0 & S_{a_4} & N_4 \\
                0 & 0 & 0 & 0 & S_{a_5}
            \end{pmatrix}_{(\sum_i{a_i}-3)\times\sum_i{a_i}},
        \]
        for arbitrary choice of $n_1$ and $n_4$. The projection of $\Lambda$ over the pentagon labeled by $(b_1,\dotsc,b_5)$ has type
        \[
        \left\{
          \begin{array}{ll}
            3, & \hbox{if $A=A_3$ and $n_{1b_1}+n_{4b_4}=0$; or $A=A_2$ and $n_{1b_1}+n_{4b_4}=1$;} \\
            2, & \hbox{if $A=A_3$ and $n_{1b_1}+n_{4b_4}=1$; or $A=A_2$ and $n_{1b_1}+n_{4b_4}=0$.}
          \end{array}
        \right.
        \]
        Therefore, the number of Davis-Januszkiewicz classes of small covers over $P_5(a_1,\dotsc,a_5)$ is
        \[
             \hbox{\#DJ}= 2^{a_1+a_4-1}+2^{a_2+a_5-1}+2^{a_3+a_1-1}+2^{a_4+a_2-1}+2^{a_5+a_3-1}-5.
         \]
    \end{theorem}
    \begin{proof}
        Every argument goes the same as that of Theorem~\ref{thm:toricoverpentagon}. The D-J equivalence type of $\Lambda$ is determined by the types of projected characteristic maps. We can choose $A$ from $A_3$ and $A_2$ and there are $2^{a_1-1}$ and $2^{a_4-1}$ choices of the vectors $n_1$ and $n_4$ respectively, thus we have $2\cdot 2^{a_1-1}\cdot 2^{a_4-1} = 2^{a_1+a_4-1}$ possibilities. Considering rotational symmetry of $P_5$, the D-J classes are all distinct except possibly the case $n_1$ and $n_4$ are zero vectors. For example, both of
        \[
            \Lambda=\begin{pmatrix}
                M_1 & M_2 & M_3 & M_4 & M_5 \\
                S_{a_1} & 0 & 0 & 0 & 0 \\
                0 & S_{a_2} & 0 & 0 & 0 \\
                0 & 0 & S_{a_3} & 0 & 0 \\
                0 & 0 & 0 & S_{a_4} & 0 \\
                0 & 0 & 0 & 0 & S_{a_5}
            \end{pmatrix} ,\text{ when } A=A_2
        \]
        and
        \[
            \Lambda'=\begin{pmatrix}
                M_5 & M_1 & M_2 & M_3 & M_4 \\
                0 & S_{a_2} & 0 & 0 & 0  \\
                0 & 0 & S_{a_3} & 0 & 0  \\
                0 & 0 & 0 & S_{a_4} & 0   \\
                0 & 0 & 0 & 0 & S_{a_5}  \\
                S_{a_1} & 0 & 0 & 0 & 0
            \end{pmatrix},\text{ when } A=A_3
        \]
        give the same types (it is 2 in this case) on every projection. Since there are five such cases, we obtain the wanted result.
    \end{proof}

    When $k=4$, the standard Gale diagram is a heptagon and $ P_{[7]} = [1,1,1,1,1,1,1]$ is a simple 4-polytope with 7 facets. Let us label the vertices of the heptagon by $1,2,5,3,4,6,7$. It is actually the dual of a cyclic polytope; see \cite{Ero08} for reference. There are only two small covers on $[1,1,1,1,1,1,1]$, which we denote by $\lambda_1$ and $\lambda_2$,
    \begin{equation}\label{eqn:charmapsoverheptagon}
        \lambda_1=\begin{pmatrix}
            1&0&0&0&1&0&1\\
            0&1&0&0&0&1&1\\
            0&0&1&0&1&1&0\\
            0&0&0&1&1&1&1
        \end{pmatrix}\text{  and  }
        \lambda_2=\begin{pmatrix}
            1&0&0&0&1&1&1\\
            0&1&0&0&1&0&1\\
            0&0&1&0&0&1&1\\
            0&0&0&1&1&1&0
        \end{pmatrix},
    \end{equation}
    up to D-J equivalence.
    For example, the following is a characteristic map for $[2,1,1,1,1,1,1]=\wed_{F_1}([1,1,1,1,1,1,1])$:
    \[
        \lambda=\begin{pmatrix}
            1_1&1_2&2&3&4&5&6&7 \\ \hline
            1&0&0&0&0&1&0&1\\
            0&1&0&0&0&1&0&1\\
            0&0&1&0&0&0&1&1\\
            0&0&0&1&0&1&1&0\\
            0&0&0&0&1&1&1&1
        \end{pmatrix}.
    \]
    Observe the first and second rows of $\lambda$ and compare it with the matrix \eqref{eqn:canonicalmatrix}.  We write
    \[
        \lambda_1=\begin{pmatrix}
            1&0&0&0& \mathbf{a_1}\\
            0&1&0&0&\mathbf{a_2}\\
            0&0&1&0&\mathbf{a_3}\\
            0&0&0&1&\mathbf{a_4}
        \end{pmatrix}
    \]
    and
    \[
        \lambda_2=\begin{pmatrix}
            1&0&0&0& \mathbf{b_1}\\
            0&1&0&0&\mathbf{b_2}\\
            0&0&1&0&\mathbf{b_3}\\
            0&0&0&1&\mathbf{b_4}
        \end{pmatrix},
    \]
    where $\mathbf{a_i}$ and $\mathbf{b_i}$ are row vectors of dimension 3.
    In order to find a nontrivial wedging over $[2,1,1,1,1,1,1]$, one must find $k$, $1\le k\le 4$, such that
    \[
        \mathbf{a_i} = \mathbf{b_i}   \text{ when } i\ne k;
    \]
    and
    \[
        \mathbf{a_i} \ne \mathbf{b_i}   \text{ when } i=k,
    \]
    but there is no such $k$. In conclusion, there are only two small covers over $[2,1,1,1,1,1,1]$ up to D-J equivalence: one is the canonical extension of $\lambda_1$ and the other is the trivial wedging of $\lambda_2$. This holds for arbitrary $P_{[7]}(J)$. Meanwhile, we recall Theorem~\ref{thm:notoricvariety} saying that there is no toric manifold over $P_{[7]}$. Hence, no small cover over $[a_1, a_2, \ldots, a_7]$ is a real toric manifold while every small cover over $[a_1, \ldots, a_5]$ is a real toric variety.

    We state this result as a proposition:
    \begin{proposition} \label{prop:realtoricvarietyoverheptagon}
        Up to Davis-Januszkiewicz equivalence, there are exactly two small covers over the simple polytope $[a_1,a_2,\dotsc,a_7]$. These cannot be real toric manifolds.
    \end{proposition}

    For a fixed simple polytope $P$, we remark there is a computer algorithm \cite[Algorithm~4.1]{GS03} to find every small cover over $P$, although the printed version in \cite{GS03} has small error on it. We present the corrected version here.

    \begin{itemize}
        \item \textbf{Input:} $FP=$ set of subsets $I\subset\{1,\ldots,m\}$ such that $\bigcap_{i\in I}F_i$ is a face of $P$.
        \item \textbf{Output:} $\Gamma=$ list of $\Z_2$-vectors $(\lambda_1,\ldots,\lambda_m)$ such that the first $n$ vectors $\lambda_1,\ldots,\lambda_n$ form the standard basis for $\Z_2^n$.
        \item \textbf{Initialization:} Set the following:

            $\lambda_1\leftarrow(1,0,\ldots,0)$, $\lambda_2\leftarrow(0,1,\ldots,0)$, $\cdots$,
            $\lambda_n\leftarrow(0,0,\ldots,1)$,

            $\Gamma\leftarrow\varnothing$,

            $S\leftarrow$ list of nonzero elements of $\Z_2^n$,

            $i\leftarrow n+1$.
        \item \textbf{Procedure:}

            \begin{enumerate}
                \item Set $S_i \leftarrow S$.
                \item For all $I\in FP$ of the form $I=\{i_1,\ldots,i_k\}\cup\{i\}$ with $1\le i_i\le\cdots\le i_k \le i$, remove the vector $\lambda_{i_1}+\cdots+\lambda_{i_k}$ from the list $S_i$.
                \item If $i=n$ then STOP.
                \item If $S_i=\varnothing$ then $i\leftarrow i-1$ and go to (3).
                \item Set $\lambda_i\leftarrow S_i[1]$ (where $S_i[1]$ denotes the first element of the list $S_i$).
                \item Remove $\lambda_i$ from the list $S_i$.
                \item If $i=m$, then add the vector to the list and go to (4).
                \item If $i<m$, then set $i\leftarrow i+1$ and go to (1).
            \end{enumerate}
    \end{itemize}

    For example, the algorithm above applied to the polytope $[1,2,1,2,2]$ gives the following list of matrices $B_i$ such that $\lambda = (I_5 \mid B_i)$ is a characteristic matrix of the polytope $[1,2,1,2,2]$ where $I_5$ is a $5\times 5$ identity matrix. Note that the facets of $[1,2,1,2,2]$ are ordered so that the first five facets intersect and therefore the first five columns forms an identity matrix. We used the ordering $1,2_1,3,4_1,5_1,2_2,4_2,5_2$.
    \[
        B_1=\begin{pmatrix}
            0 & 1 & 1 \\
            1 & 0 & 0 \\
            1 & 1 & 0 \\
            0 & 1 & 0 \\
            0 & 0 & 1
        \end{pmatrix},\:
        B_2=\begin{pmatrix}
            0 & 1 & 1 \\
            1 & 0 & 0 \\
            1 & 1 & 1 \\
            0 & 1 & 0 \\
            0 & 0 & 1
        \end{pmatrix},\:
        B_3=\begin{pmatrix}
            0 & 1 & 1 \\
            1 & 0 & 1 \\
            1 & 1 & 0 \\
            0 & 1 & 0 \\
            0 & 0 & 1
        \end{pmatrix},\:
        B_4=\begin{pmatrix}
            0 & 1 & 1 \\
            1 & 0 & 1 \\
            1 & 1 & 1 \\
            0 & 1 & 0 \\
            0 & 0 & 1
        \end{pmatrix},\:
        \]
    \[
        B_5=\begin{pmatrix}
            0 & 1 & 1 \\
            1 & 0 & 0 \\
            1 & 1 & 1 \\
            0 & 1 & 0 \\
            1 & 0 & 1
        \end{pmatrix},\:
        B_6=\begin{pmatrix}
            0 & 1 & 1 \\
            1 & 0 & 0 \\
            1 & 1 & 0 \\
            0 & 1 & 0 \\
            1 & 1 & 1
        \end{pmatrix},\:
        B_7=\begin{pmatrix}
            0 & 1 & 1 \\
            1 & 0 & 0 \\
            1 & 1 & 1 \\
            1 & 1 & 0 \\
            0 & 0 & 1
        \end{pmatrix},\:
        B_8=\begin{pmatrix}
            0 & 1 & 1 \\
            1 & 0 & 0 \\
            1 & 1 & 1 \\
            1 & 1 & 0 \\
            1 & 0 & 1
        \end{pmatrix},\:
    \]
    \[
        B_9=\begin{pmatrix}
            1 & 0 & 1 \\
            1 & 0 & 0 \\
            0 & 1 & 1 \\
            0 & 1 & 0 \\
            0 & 0 & 1
        \end{pmatrix},\:
        B_{10}=\begin{pmatrix}
            1 & 0 & 1 \\
            1 & 1 & 0 \\
            0 & 1 & 1 \\
            0 & 1 & 0 \\
            0 & 0 & 1
        \end{pmatrix},\:
        B_{11}=\begin{pmatrix}
            1 & 1 & 1 \\
            1 & 0 & 0 \\
            0 & 1 & 1 \\
            0 & 1 & 0 \\
            0 & 0 & 1
        \end{pmatrix},\:
        B_{12}=\begin{pmatrix}
            1 & 1 & 1 \\
            1 & 1 & 0 \\
            0 & 1 & 1 \\
            0 & 1 & 0 \\
            0 & 0 & 1
        \end{pmatrix},\:
    \]
    \[
        B_{13}=\begin{pmatrix}
            1 & 1 & 1 \\
            1 & 0 & 0 \\
            0 & 1 & 1 \\
            0 & 1 & 0 \\
            1 & 0 & 1
        \end{pmatrix},\:
        B_{14}=\begin{pmatrix}
            1 & 0 & 1 \\
            1 & 0 & 0 \\
            0 & 1 & 1 \\
            1 & 1 & 1 \\
            0 & 0 & 1
        \end{pmatrix},\:
        B_{15}=\begin{pmatrix}
            1 & 1 & 1 \\
            1 & 0 & 0 \\
            0 & 1 & 1 \\
            1 & 1 & 0 \\
            0 & 0 & 1
        \end{pmatrix},\:
        B_{16}=\begin{pmatrix}
            1 & 1 & 1 \\
            1 & 0 & 0 \\
            0 & 1 & 1 \\
            1 & 1 & 0 \\
            1 & 0 & 1
        \end{pmatrix},\:
    \]
    \[
        B_{17}=\begin{pmatrix}
            1 & 0 & 1 \\
            1 & 0 & 0 \\
            1 & 1 & 0 \\
            0 & 1 & 0 \\
            0 & 0 & 1
        \end{pmatrix},\:
        B_{18}=\begin{pmatrix}
            1 & 0 & 1 \\
            1 & 0 & 0 \\
            1 & 1 & 0 \\
            0 & 1 & 0 \\
            1 & 1 & 1
        \end{pmatrix},\:\text{and }
        B_{19}=\begin{pmatrix}
            1 & 0 & 1 \\
            1 & 0 & 0 \\
            1 & 1 & 0 \\
            1 & 1 & 1 \\
            0 & 0 & 1
        \end{pmatrix}.
        \]
    Using this data, we have an alternative way to find all characteristic maps over $[a_1,a_2+1,a_3,a_4+1,a_5+1]$ ($a_2$, $a_4$, and $a_5$ are allowed to be zero). For a suitable basis, its characteristic matrix has the form
    \[
        \left(
            \begin{array}{ccc|ccc|ccc|ccc}
              & & & & & & & & &  & A_{11} &  \\
              & I_{n_1} & & & 0 & & & 0 & &  & \vdots &        \\
              & & & & & & & & &  & A_{1n_1} &  \\ \hline
              & & & & & & & & &   &   &   \\
              &0 & & &\ddots & & & 0& &   & \vdots &   \\
              & & & & & & & & &   &   &   \\ \hline
              & & & & & & & & &   & A_{51} &   \\
              &0 & & &0 & & & I_{n_5} & &   & \vdots &   \\
              & & & & & & & & &   & A_{5n_5}  &
            \end{array}
        \right),
    \]
    where $A_{i,j}$ are 3-dimensional row vectors such that for any $j_i$'s, the matrix
    \[
        \begin{pmatrix}
            A_{1j_1} \\
            \vdots  \\
            A_{5j_5}
        \end{pmatrix}
    \]
    is one of $B_i$'s above.

\section{Application: the Lifting problem}\label{sec:liftingproblem}
    Recall that every toric manifold has its conjugation map. Like toric manifolds, every topological toric manifold $M$ as a $T^n$-manifold has a conjugation map whose fixed points make a real topological toric manifold $M'$ as a $(\Z_2)^n$-manifold. In this case, $\lambda(M')$ is exactly the modulo 2 reduction of $\lambda(M)$. Hence, it seems natural to ask whether the converse holds or not. From this viewpoint, L\"{u} presented the following problem, so called the \emph{lifting problem}, at the conference on toric topology held in Osaka in November 2011.\footnote{\url{http://www.sci.osaka-cu.ac.jp/~masuda/toric/torictopology2011_osaka.html}}

    \begin{question}[Lifting problem for real topological toric manifolds]
    Let $K$ be a fan-like simplicial sphere of dimension $n-1$ with $m$ vertices.
    Let $M$ be a real topological toric manifold over $K$. Then, is there a topological toric manifold $N$, called a \emph{lifting} of $M$, such that $M$ is the fixed point set of the conjugation on $N$? Equivalently, for any non-singular characteristic map $\lambda \colon V(K) \to \Z_2^n$ over $\Z_2$, is there a non-singular characteristic map $\widetilde{\lambda}$ over $\Z$, called a \emph{lifting} of $\lambda$, such that
    $$
        \xymatrix{
            & \Z^n \ar[d]^{\text{mod $2$}}\\
        V(K)\ar[ur]^{\widetilde{\lambda}} \ar[r]_{\lambda} & \Z_2^n,
        }
    $$ where $V(K)$ is the vertex set of $K$?
    \end{question}

    When $n \leq 3$, it is known that the answer to the lifting problem is affirmative. In this paper, we answer to the problem affirmatively when $m \leq n+3$.

    \begin{lemma} \label{lem:linearalgebra}
        Let $A = (a_{ij})_{n \times n}$ be an $n \times n$ matrix with integer entries such that $\det A$ is odd. Then, there is an $n \times n$ matrix $B = (b_{ij})_{n \times n}$ such that $\det B = 1$ and $b_{ij}$ is congruent to $a_{ij}$ up to modulo $2$ for all $i$ and $j$.
    \end{lemma}
    \begin{proof}
        We use an induction on $n$. It is obvious for $n=1$. Let $A$ be an $n\times n$ matrix. Let us denote by $A_{ij}$ the minor of $A$ obtained by deleting the $i$-th row and the $j$-th column of $A$. Then recall
        \[
            \det A = \sum_{i=1}^n a_{1i}\det A_{1i}.
        \]
        One can assume that $\det A_{11} = 1$ by induction hypothesis. Replace $a_{11}$ with $a_{11} - \det A +1$ to obtain $B$ which is available since $1-\det A $ is even. Then $\det B = 1$ and the proof is done.
    \end{proof}

    The above lemma says that it is enough to consider the D-J equivalent class of toric objects for the lifting problem. To be more precise, let $\lambda$ be a characteristic map over $\Z_2$ of dimension $n$. Then every characteristic map over $\Z_2$ which is D-J equivalent class is given by the matrix $R\lambda$ when $R $ is an $n\times n$ matrix over $\Z_2$ whose determinant is $1$. If $\widetilde{\lambda}$ is a lifting of $\lambda$, then the above lemma guarantees that there is an $n \times n$ matrix $\widetilde{R}$ with $\det \widetilde{R} = 1$ such that $R$ is the modulo 2 reduction of $\widetilde{R}$ and therefore $\widetilde{R}\widetilde{\lambda}$ is a characteristic map D-J equivalent to $\widetilde{\lambda}$.

    \begin{corollary}
       Let $K$ be a fan-like simplicial sphere of dimension $n-1$ with at most $n+3$ vertices.
      Then any real topological toric manifold over $K$ can be realized as fixed points of the conjugation of a topological toric manifold.
    \end{corollary}

    \begin{proof}
        By \cite{CMS10}, any real topological toric manifold $M$ over the join of boundaries of simplices is a generalized real Bott manifold which is real toric variety. Hence, there is a (projective) toric manifold whose fixed point set of the conjugation is $M$. Indeed, this toric manifold is a generalized Bott manifold.

        Hence, it is enough to consider the case when the number of vertices is $n+3$, and the Gale diagram of $P$ is either a pentagon or a heptagon where $\partial P^* = K$.
        If $P = P_5(a_1, \ldots, a_5)$, then, by Theorem~\ref{thm:toricoverpentagon} and Theorem~\ref{thm:realtoricoverpentagon}, our claim holds.
        If $P = [a_1, \ldots, a_7]$, then, by Proposition~\ref{prop:realtoricvarietyoverheptagon}, every small cover over $P$ can be obtained by canonical extensions from either $M(P_{[7]}, \lambda_1)$ or $M(P_{[7]}, \lambda_2)$ where $\lambda_1$ and $\lambda_2$ are given by \eqref{eqn:charmapsoverheptagon}. If we regard both $\lambda_1$ and $\lambda_2$ as $(0,1)$-matrix over $\Z$, then one can check that $\lambda_1$ and $\lambda_2$ are non-singular over $\Z$. Hence, so are their canonical extensions, which proves the corollary.
    \end{proof}

\section*{Acknowledgement}
    When the first author was visiting the Institute for Advanced Study (IAS) and the Rider university in February, 2012, he had a useful discussion with Professor Tony Bahri which inspired and stimulated this work. The authors would like to thank to him.

\end{document}